\documentclass[leqno]{article}
\usepackage[T1]{fontenc}
\usepackage[utf8]{inputenc}
\usepackage{float}
\usepackage{upgreek}
\usepackage{amsfonts,amssymb,amsthm,latexsym}
\usepackage[bbgreekl]{mathbbol}
\DeclareSymbolFontAlphabet{\mathbb}{AMSb}
\DeclareSymbolFontAlphabet{\mathbbl}{bbold}
\usepackage{etextools}
\usepackage{mathtools} % also load amsmath package
\usepackage{textcomp}
\usepackage{relsize}
\usepackage{ifthen}
\usepackage{mathabx}
\usepackage{eucal}
\usepackage{indentfirst}
\usepackage{tikz}
\usetikzlibrary{calc}
\usetikzlibrary{arrows,decorations.pathmorphing,decorations.markings,backgrounds,positioning,fit,petri,arrows.meta} 
\usetikzlibrary{patterns, patterns.meta,math}
\usepackage[scr]{rsfso}
\usepackage{graphics}
\usepackage{bbm} % for characteristic function 1 mathbb
\usepackage{anyfontsize}
\usepackage{stmaryrd}
\newcommand{\lsum}{\mathlarger{\mathlarger{\mathlarger{\sum}}}}
\newcommand{\absarg}{\left|\,\underline{\;}\,\right|}
\newcommand{\dotprod}{\bullet}

\newcommand{\primeset}{\text{Prime}}
\newcommand{\fg}{\mathcal{G}} 

\newcommand{\mtx}[4]{\begin{pmatrix} \scriptstyle #1  & \ \hspace{-.3cm} \scriptstyle #2
\vspace{-.15cm}\\ \scriptstyle #3  &\ \hspace{-.3cm} \scriptstyle  #4 \end{pmatrix}}
\newcommand{\xtworightarrow}[2][]{%
				  \xrightarrow[#1]{#2}\mathrel{\mkern-14mu}\rightarrow
	}

\newcommand{\op}{\text{Op}}

\makeatletter
\newcommand*{\rom}[1]{\expandafter\@slowromancap\romannumeral #1@}
\makeatother
\newcommand{\unit}{{\mathbf{1}}}
\newcommand{\Unit}{{\underline{1}}}

\newcommand{\F}{\mathbb{F}}
\newcommand{\A}{\mathbb{A}}
\newcommand{\Q}{\mathbb{Q}}
\newcommand{\R}{\mathbb{R}}
\newcommand{\C}{\mathbb{C}}
\newcommand{\N}{\mathbb{N}}
\newcommand{\GL}{\mbox{GL}}

\makeatletter
\DeclareRobustCommand{\properideal}{\mathrel{\text{$\m@th\proper@ideal$}}}
\newcommand{\proper@ideal}{%
  \ooalign{$\lneq$\cr\raise.22ex\hbox{$\lhd$}\cr}%
}
\makeatother
\newcommand{\lideal}{\lhd}
\newcommand{\lidealeq}{\unlhd}

\newcommand{\Z}{\mathbb{Z}}
\DeclareMathOperator{\cor}{Cor} 
\DeclareMathOperator{\real}{ Re\,}
\DeclareMathOperator{\area}{Area\,}
\DeclareMathOperator{\cl}{Cl\,}
\DeclareMathOperator{\SL}{SL\,}

\DeclareMathSymbol{\llcurly}{\mathrel}{mathb}{"CE}
\DeclareMathSymbol{\ggcurly}{\mathrel}{mathb}{"CF}

\newtheorem*{claim*}{Claim}
\newtheorem{theorem}{Theorem}
\newtheorem{definition}[equation]{Definition} % [equation] counter instead of [section]
\newtheorem{definition*}{Definition} 
\newtheorem{remark}[equation]{Remark}
\newtheorem*{remark*}{Remark}

\newtheorem*{corollary*}{Corollary}
\newtheorem{proposition}{Proposition}
\numberwithin{equation}{section}
\title{Meditations on the Farey Fractal}
\author{Shai Haran \\ {haran@technion.ac.il} }
\date{}
\begin{document}

\nocite{*} 			% Print all citation regardless whether the have been cited or not. To make it
								% appear in the same order (in bib file) we put it close to begin{document}
\maketitle
\begin{abstract}
				We define the ``coronas'', which are especially spikey paths in the Farey graph
				going from $\underline{0}=(1,0)$ to $\underline{\infty}=(0,1)$. We show that for
				$R\ge 2$, $\left\{ (x,y)\in \N^{+}\times\N^{+},\; \gcd(x,y)=1, \; x+y\le R
				\right\}$ is a corona.
\end{abstract}
\section*{Preface}
For me, the greatest mystery of mathematics was Andr\'e Weil's ``Roseta Stone''  \cite{W1939}, \cite{W1940a}: the
analogies between number fields and function fields. \\
Regarding the Riemann hypothesis, initially I followed Weil's approach 
 \cite{W1966} to Tate's
 thesis \cite{T1950}, viewing it as the harmonic analysis of the action of $\A^{\star}_{K}/K^{\star}$ on
 $\A_{K}/K^{\star}$ (or on distribution on $\A_{K}$ that are $K^{\star}$-invariant), and
 on $\mathbb{P}^{1}\left( \mathbb{A}_{k} \right)/K^{\star}$. This
suggested that an (real valued) index-theorem, analogue of the Riemann-Roch for the associated surface,
will give a proof of the Riemann Hypothesis along the lines of Weil's proof, see
\cite{H1989}. The formula \cite{H1990} for Weil's explicit-sums-distribution \cite{W1952} was also the
starting point for the program of Alain Connes and collaborates (cf. \cite{C1999}
appendix, where the formula of \cite{H1990} is stated in an assymptotic form). But note
that there is still not even a proof of Weil's Riemann-Hypothesis for a function field
$K$ \cite{W1941} using the ``non-commutative'' space $\A_{K}/K^{\star}$! \vspace{.1cm}\\
The mysterious analogy
between number fields and function fields is clarified by the concept of generalized-ring
(see \cite{H2010} for a quick introduction). The language of generalized-rings can be used
as the foundation of algebraic geometry in the style of Grothendieck (see \cite{H2017}):
\begin{itemize}
				\item The final object of geometry is the \underline{absolute-point}
								$\text{spec}(\mathbb{F})$, where $\mathbb{F}$ is the initial object of
								generalized-rings, the \underline{``Field with one element''}
				\item The real $\R$ and complex $\C$ numbers, when viewed as (topological) \break
								generalized-rings, have (maximal compact topological)-sub-generalized-rings
								$\Z_{\R}\subseteq \R$, and $\Z_{\C}\subseteq \C$, (analogous to $\Z_{p}\subseteq
								\Q_{p}$); and $\text{spec}(\Z)$, and $\text{spec}(O_{K})$, $K$ a
								number field, have natural compactifications. 
				\item There are non-trivial Arithmetical surfaces, and higher arithemetical
								dimensions, as the tensor-product ($=$the categorical sum) does not reduce
								to its diagonal: 
								\begin{equation*}
												\Z\otimes_{\mathbb{F}}\Z \neq \Z.
								\end{equation*}
				\item There is a natural generalization of homological algebra, and of the derived
								category of quasi-coherent sheaves of $O_{X}$-modules,
								and the derived functors of direct and inverse images, \cite{H2020}. \vspace{.2cm}\\ 
\end{itemize}
								However, we are still missing an arithmetical analogue of
				the \underline{Frobenius}\break \underline{correspondence}. 
								For the \underline{tropical}
								examples (see \cite{H2010} and \cite{H2017}, p.29):
								\begin{equation*}
												\mathcal{B} = \left\{ 0,1 \right\}_{t} \subseteq
												\mathcal{J}=[0,1]_{t}\subseteq\mathcal{R}=[0,\infty)_{t}
								\end{equation*}
								where the subscript ``$t$'' indicates that addition is $x+y:=\max\left\{
								x,y \right\}$, we have that $\mathcal{R}$ is a generalized-field, and the
								(multiplicative) group $\R^{+}$ acts on $\mathcal{R}$ by automorphism
								$x\mapsto x^{p}$, with fixed field the \underline{Boolean-field}
								$\mathcal{B}$. This resembles the Frobenius-automorphism $x\mapsto
								x^{p}$ of the field $\overline{\mathbb{F}_{p}}$ with fixed field
								$\mathbb{F}_{p}$. Unfortunately, the tensor-product (=categorical sum, in
								the categories of generalized-rings or of semi-rings) vanishes:
								$\Z\otimes\mathcal{B}=\left\{ 0 \right\}$ the zero=final object.
								\vspace{.2cm}\\ 
								As Weil
								suggested \cite{W1942}, the arithmetical analogue of extending scalars to
								$\overline{\mathbb{F}_{p}}$, is the cyclotomic extension obtained by
								adding (all!) roots of unity $\bbmu_{\infty}$ (this idea, in the $p$-power cyclotomic
												extension, $\Z\left[ \mu_{p^{\infty}} \right]$, was developed by K. Iwasawa, who related it to the
								Kobuta-Leopoldt $p$-adic $L$-function \cite{Iw1972}). \vspace{.2cm}\\
								All this makes the prospect of seeing, in our life-time, a proof of the
												Riemann-Hypothesis, along the lines of Weil's proof for a function
												field, unrealistic! But perhaps there is an alternative route:
												after all, the field of rational numbers $\Q$ is the analogue of
												the field of rational functions $\mathbb{F}_{p}(T)$, and the
												Riemann-Hypothesis for $\F_{p}(T)$, that is for
												$\mathbb{P}^1/\F_{p}$, is a triviality: there are no zeros
												of the zeta function, only the poles, and we can
												\underline{constructively}
												generate all the primes,
												$\overline{\mathbb{F}}_{p}/\left(x\hspace{-.15cm}\sim\hspace{-.15cm}
								x^{p}\right)$, and we can count
												exactly the number of points, 
												\begin{equation*}
																\# \mathbb{P}^{1}\left( \mathbb{F}_{p^{d}} \right) =
																\frac{p^{2d}-1}{p^{d}-1}=p^{d}+1
												\end{equation*}
												(and the Riemann-Hypothesis for a general function field follows
												from this by the Bombieri-Stepanov argument). So perhaps, for the
												basic number field of rational numbers $K=\Q$ there is a
												\underline{constructive} root. 
%%%%%%%%%%%%%%%%%%%%%%%%%%%%%%%%%%%%%%%%%%%%%%%%%%%%%%%%%%%%%%
\section{Introduction}
The non-zero natural numbers, $\N^{+}=\N\setminus \left\{ 0 \right\}$, are the free
commutative, unital, monoid on the set of \underline{primes},
\begin{equation}
				\primeset = \N^{++}\setminus \left( \N^{++}\bullet \N^{++} \right)\quad , \quad
				\N^{++} = \N^{+} \setminus \left\{ 1 \right\}, 
				\label{eq:1}
\end{equation}
Writing $\left\{ \alpha \right\}$ for $1$ unconditionally, and for $\alpha$ iff we have
the \underline{Riemann Hypothesis}, we have 
\begin{equation}
				\frac{1}{\zeta(s)} = \prod\limits_{p\in\primeset} \left( 1-p^{-s} \right) =
								\sum\limits_{n\ge 1} \frac{\mu(n)}{n^{s}} \quad \text{Converges for
				$\real(s)>\left\{ \frac{1}{2} \right\}$}
				\label{eq:2}
\end{equation}
This is equivalent to the statement that 
\begin{equation}
				\displaystyle
				\left| \sum\limits_{n=1}^{R} \mu(n)\right| = \left| \sum\limits_{
								\begin{array}[H]{c}
												1\le x \le y \le
												R \\ \gcd(x,y) = 1
				\end{array}
				}
				\displaystyle
				e^{2\pi i {x}/{y}} \right| = O \left( R^{ \left\{ 1/2 \right\}  } \log\, R \right). 
				\label{eq:3}
\end{equation}
Similarly, the positive rational numbers $\Q^{+}$ are the free abelian group on the primes 
\begin{equation}
				\Q^{+} = \Z\; \primeset = \bigoplus_{p\in \primeset} p^{\Z}
				\label{eq:4}
\end{equation}
While this does not determines the set of primes as in (\ref{eq:1}), we do have that every $v\in \Q^{+}$ can be written uniquely as $v=y/x$ with $x,y\in \N^{+}$, and
$\gcd(x,y) = 1 $; we write $\underline{v}=(x,y)$, so that
\begin{equation}
				\underline{\Q}^{+}\equiv \left\{ (x,y)\in \N^{+},\gcd (x,y) = 1 \right\} \equiv (N^{+}\times
				\N^{+})\setminus N^{++}\bullet (\N^{+}\times \N^{+})
				\label{eq:5}
\end{equation}
We have now quite similarly to (\ref{eq:3}), with $\Phi (n)= \#\left(\Z /{n\Z}
\right)^{*}$, 
\begin{equation}
				\begin{array}[H]{lll}
				\left| \sum\limits_{n=1}^{R} \Phi(n)\right| &=& 1 + \# \left\{ (x,y)\in
				\underline{\Q}^{+}, x+y \le R \right\} \\\\
				&=&  1 + \sum\limits_{n\ge 1} \mu(n)\cdot \# \left\{ (x,y)\in \N^{+}\times
				\N^{+}, x+y\le\frac{R}{n} \right\}\\\\
				&\approx& \sum\limits_{n\ge 1} \mu(n)\cdot \frac{1}{2}\left( \frac{R}{n} \right)^2 \\\\
&=& \frac{R^2}{2\zeta(2)}+O\left( R^{ \left\{ \frac{1}{2} \right\}} \log R  \right)
				\end{array}
				\label{eq:6}
\end{equation}
Indeed, on the analytic side we have, since $\Phi$ is multiplicative 
\begin{equation}
				\begin{array}[H]{lll}
				\sum\limits_{n\ge 1}\frac{\Phi(n)}{n^{s}} &=& \prod\limits_{p\in\primeset} \left(
								1+(1-p^{-1})\sum\limits_{n\ge 1}p^{n(1-s)} 
				\right) \\\\ 
				&=& \prod\limits_{p\in\primeset} \frac{1-p^{-s}}{1-p^{i-s}} = 
				\frac{\zeta(s-1)}{\zeta(s)}. 
				\end{array}
				\label{eq:7}
\end{equation}
This has a simple pole at $s=2$, with residue $\frac{1}{\zeta(2)}$, and otherwise is
analytic for $\real(s)> \left\{ \frac{1}{2} \right\}$. \vspace{.1cm}\\  The functions $\mu(n)$ and
$\Phi(n)$ are as mysterious as the primes, e.g. $\primeset \equiv \left\{  n\in
				\N^{+}, \Phi(n) = n-1
\right\}$. But the sum in (\ref{eq:6}) is better than the sum in (\ref{eq:3}), because it can
be made \underline{constructive}. Our purpose here is to \underline{linearize} this constructive approach so as to
have explicit recursive formula for the sum in (\ref{eq:6}) and to explore some of the
structures behind it. Curiously, there is some kind of interaction between the binary and
the Fibonacci expansions of integers. 
\section{The Farey Graph}
The group $\Q^{+}$ embeds as a dense subroup
of the multiplicative group of positive real numbers 
\begin{equation}
				\Q^{+} \hookrightarrow\R^{+}= (0,\infty) \subseteq [0,\infty]
				\label{eq:8}
\end{equation}
and we have an induced total order $\le$ on $\underline{\Q}^{+}$. We add to
$\underline{\Q}^{+}$ the points in the plane 
\begin{equation}
				\underline{0}:= (1,0)\qquad , \qquad \underline{\infty}:= (0,1),
\quad
				 % \quad 
				\label{eq:9}
\end{equation}
and we have the \underline{Farey Graph} $\fg$ with vertices $\fg_0 = \left\{ \underline{\infty} \right\} \amalg %\upmodels
				\underline{\Q}^{+}\amalg %\upmodels
				\left\{ \underline{0} \right\}$ and edges
\begin{equation}
				\fg_1=\left\{ 
								\begin{array}[h]{l}
												\left( v_{-}=(x_{-},y_{-}), v_{+}=(x_{+},y_{+})\right)\in
												\fg_0\times \fg_0 \;, \\ \det \begin{pmatrix} v_{-} \\ v_{+}
				\end{pmatrix} = x_{-}y_{+}-y_{-}x_{+}=1
				\end{array}\right\}
				\label{eq:10}
\end{equation}
Every edge $(v_{-},v_{+})\in \fg_1$, gives a \underline{parallelogram} 
\begin{equation}
				\begin{array}[H]{ll}
				P_{v} = \left\{ t_{+}v_{+}+t_{-}v_{-}; t_{\pm}\in [0,1] \right\} \\
				\area\left( P_{v} \right) = \det 
				\begin{pmatrix}
								v_{-} \\ v_{+} 
				\end{pmatrix} =1 \quad , \quad P_{v} \cap (\N\times \N) = \left\{ (0,0),
				v_{-},v,v_{+} \right\} 
				\end{array}
				\label{eq:11}
\end{equation}
and a \underline{triangle}
\begin{equation}
				\begin{array}[H]{l}
				\Delta_{v} = \left\{ t_{+}v_{+}+t_{-}v_{-}; \; t_{\pm}\in [0,1]\; , \;
				t_{+}+t_{-}\ge 1 \right\} \\ 
				\area \left( \Delta_v \right)=\frac{1}{2} \quad , \quad \Delta_{v}\cap
				\left( \N\times\N \right) = \left\{ v_{+}>v>v_{-} \right\}
				\end{array}
				\label{eq:12}
\end{equation}
where we denote them using the \underline{mediant}
\begin{equation}
				v=v_{+}+v_{-} = \left( x_{+}+x_{-}, y_{+}+y_{-} \right) 
				\label{eq:13}
\end{equation}
(Obtained as the vector addition in the plane; not addition in $\Q^{+}$!). \\ 
If we add to
the triangles $\left\{ \Delta_{v} \right\}_{v\in \Q^{+}}$ the triangle 
\begin{equation}
				\Delta_{0} = \left\{ (t_{+},t_{-}); \; t_{\pm}\in[0,1],\; t_{+}+t_{-}\le 1 \right\}
				\label{eq:14}
\end{equation}
we get a triangulation of the first quadrant of the plane minus the multiples $t\cdot v$,
$v\in \underline{\Q}^{+}\amalg \left\{ \underline{0},\underline{\infty} \right\}$,
$t>1$: 
\begin{equation}
				[0,\infty)\times [0,\infty)\, \setminus\, (1,\infty)\dotprod \left(
				\underline{\Q}^{+}\amalg \left\{ \underline{0},\underline{\infty} \right\}
\right) \equiv \Delta_{0} \amalg 
\coprod\limits_{v\in\underline{Q}^{+}}\Delta_{v}
				\label{eq:15}
\end{equation}
\section{The binary tree}
The monoid 
\begin{equation}
\begin{array}[H]{lll}
				\SL_2(\N)&:=& \left\{ 
								\begin{pmatrix}
												a & b \\
												c & d
								\end{pmatrix}\in \SL_2(\Z),\; a,b,c,d\ge 0
				\right\} \\\\
				&\equiv & \left\{ 
								\begin{pmatrix}
												v_{-} \\
												v_{+}
								\end{pmatrix}, \; (v_{-},v_{+})\in \fg_1
				\right\} \equiv \fg_{1}
				\end{array}
				\label{eq:16}
\end{equation}
acts on $\underline{\Q}^{+}$, (and on $\underline{\Q}^{+}\amalg \left\{
\underline{0},\underline{\infty} \right\})$, on the right: 
				\begin{equation}
								\begin{array}[H]{lll}
								(x,y)
								\begin{pmatrix}
												v_{-}\\ v_{+}
								\end{pmatrix} = x\cdot v_{-}+y\cdot v_{+} \\\\
								\underline{0} 
								\begin{pmatrix}
												v_{-} \\ v_{+}
								\end{pmatrix}= v_{-} \qquad , \qquad \underline{\infty}
								\begin{pmatrix}
												v_{-} \\ v_{+}
								\end{pmatrix} = v_{+}
								\end{array}
								\label{eq:17}
				\end{equation}
				This action preserves the graph structure $\fg$, and the triangles
				$\Delta_{v}$: 
				\begin{equation}
								\Delta_{v}\cdot g = \Delta_{vg}\quad , \quad
								v\in\underline{\Q}^{+}\quad , \quad g\in \SL_2(\N).
								\label{eq:18}
				\end{equation}
				For $g= 
				%
								\label{eq:19}
				\end{equation}
				i.e. either ($b\ge a$ and $d\ge c$) \underline{or} ($b\le a$ and $d\le c$). \\ 
				It follows that $\SL_2(\N)$ is the \underline{free} monoid on these two generators 
				\begin{equation}
								\SL_{2}(\N) = \left\langle g_{+} = 
								%
								\right\rangle
								\label{eq:20}
				\end{equation}
				and every $g\in \SL_{2}(\N)$ has a unique representation as a
				``\underline{word}'' 
				\begin{equation}
								\begin{array}[H]{l}
								g = g_{\delta}^{a_\ell}\dots g_{+}^{a_{2}}g_{-}^{a_{1}}g_{+}^{a_0}\quad ,
								\quad a_0\ge 0, \quad a_{1+i}\ge 1 \quad , \quad 
								\delta = \delta(g) = (-1)^{\ell}
								\end{array}
								\label{eq:21}
				\end{equation}
				The action of $\SL_{2}(\N)$ on $\underline{\Q}^{+}$ is free, and we get an
				\underline{identification} 
				\begin{equation}
								\begin{array}[H]{rcl}
								\SL_2(\N) &\xleftrightarrow{\qquad \sim \qquad}& \underline{\Q}^{+} \\\\
								\begin{pmatrix}
												v_{-} \\
												v_{+}
								\end{pmatrix} &\xleftrightarrow{\quad\quad} & v \\\\
								g= 
								\begin{pmatrix}
												v_{-} \\
												v_{+}
								\end{pmatrix}  
								 &\xmapsto{\quad \quad}& v=(1,1)g = v_{+}+v_{-}  \\\\
								g_{v} = 
								\begin{pmatrix}
												v_{-} \\ v_{+} 
								\end{pmatrix} & \tikz \draw[arrows={stealth[scale=.5]-|[scale=.5]}]
								(0,0)--(.9,0); & v  %\longmapsfrom 
								\end{array} 
								\label{eq:22}
				\end{equation}
Under this identification we have
\begin{equation}
				g_{v}=g_{\delta}^{a_{\ell}}\cdots
				g_{-}^{a_1}g_{+}^{a_0}\xleftrightarrow{\quad\quad} v=(x,y)
				\label{eq:23}
\end{equation}
iff we have the \underline{continued fraction expansion}
\begin{equation}
				%
				\label{eq:24}
\end{equation}
We get the maps of \underline{upper} and \underline{lower bounds}: 
\begin{equation}
				\begin{array}[H]{l}
				t_{+}: \underline{\Q}^{+} \xrightarrow{\quad\quad}\underline{\Q}^{+} \amalg \left\{
\underline{\infty} \right\}, \quad t_{+}(v) = v_{+} = \underline{\infty}\; g_v \\\\
t_{-}: \underline{\Q}^{+} \xrightarrow{\quad\quad} \underline{\Q}^{+}\amalg \left\{
				\underline{0}
\right\}, \quad t_{-}(v)= v_{-} = \underline{0}\; g_{v} \\\\
t_{+}^{-1}(\underline{\infty}) = \left\{ (1,n), n\ge 1 \right\} =
f_{\underline{\infty}}^{-}, \quad t_{-}^{-1}(\underline{0}) = \left\{ (n,1), n\ge 1
\right\} = f_{\underline{0}}^{+}
				\end{array}
				\label{eq:25}
\end{equation}
We get a structure of a binary tree on $\underline{\Q}^{+}$, the Stern-Brocot tree, with root
$\underline{1}=(1,1)$, and each $v\in \underline{\Q}^{+}$ has the two
\underline{offsprings}
\begin{equation}
				\begin{array}[H]{l}
\begin{tikzpicture}
				\coordinate [label=left:$v$] (A) at (-1cm,-1cm);
				\coordinate [label=right:$v+v_{+}$] (B) at (0.5cm,-0.2cm);
				\coordinate [label=right:$v+v_{-}$] (C) at (0.5cm,-2cm);
				\draw[->] (A)--(B);
				\draw[->] (A)--(C);
\end{tikzpicture}
				% v \tikz \draw (0pt,0pt) ->(20pt,6pt) ;
				\end{array}
				\label{eq:26}
\end{equation}
In terms of the identification (\ref{eq:22}), this can be written as 
\begin{equation}
				%
 \\
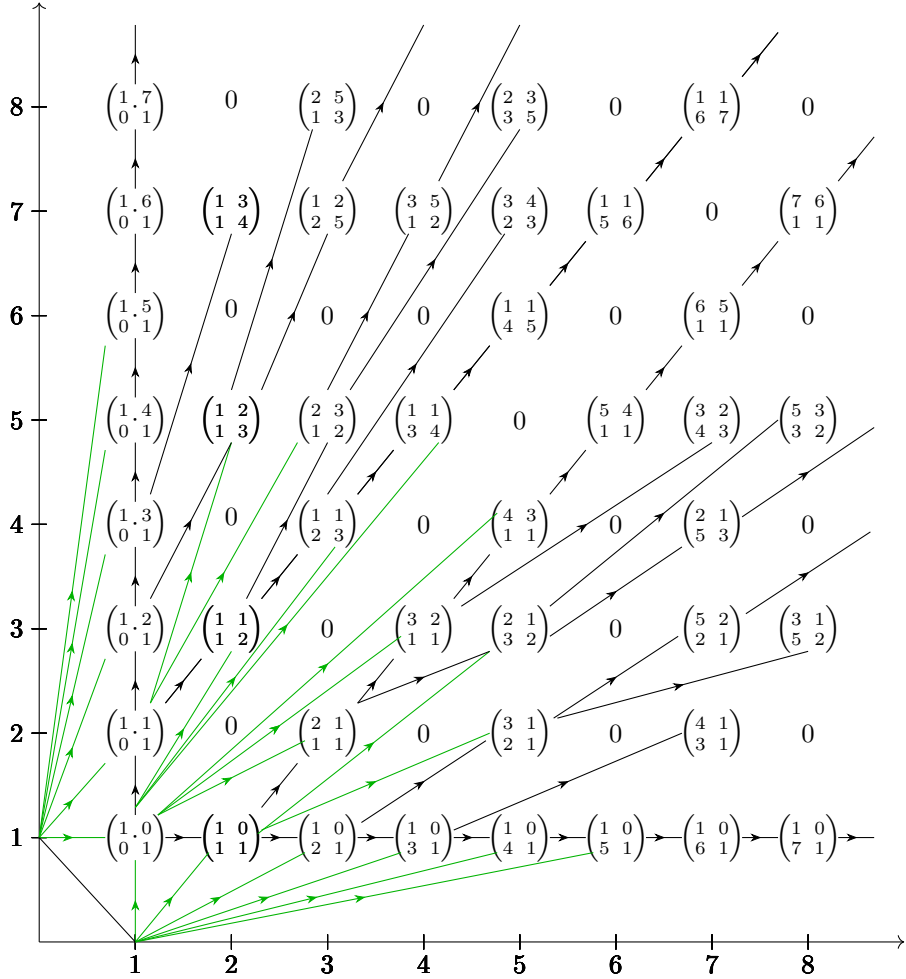
				\caption{The Stern-Brocot tree}
				\label{fig:1}
\end{figure} \ \\
% \pgfmathparse{sin(30), hex(16)}\pgfmathresult
Thus each $v\in \underline{\Q}^{+}\setminus \left\{ (1,1) \right\}$ has a unique
\underline{Mother} given by 
\begin{equation}
				%
				\label{eq:28}
\end{equation}
We also define the \underline{Father} of $v\in \underline{\Q}^{+}\setminus \left\{
				(1,1)
\right\}$ to be 
\begin{equation}
				F(v) = t_{\delta(v)}(v)= \left\{
								\begin{array}[H]{cc}
												v_{+} & \delta(v)=+ \\\\
												v_{-} & \delta(v)=- 
				\end{array} \right\}\in \underline{\Q}^{+}\amalg \left\{
												\underline{0},\infty \right\}
				\label{eq:29}
\end{equation}
We have for all $v\in \underline{\Q}^{+}\setminus \left\{ (1,1) \right\}$, 
\begin{equation}
				\begin{array}[H]{c}
				\left\{ F(v),M(v) \right\} \equiv \left\{ v_{+},v_{-} \right\} \\\\
				\text{if } \delta(v) =\; +,\;\text{"convex"}:\;\; F(v)=v_{+}, \; M(v)=v_{-} \\\\
				\text{if } \delta(v) =\; - , \;\text{"concave"}:\;\; F(v)=v_{-}, \; M(v)=v_{+} \\\\
				\end{array}
				\label{eq:30}
\end{equation}
For the root $v=\underline{1}=(1,1)$, we have $g_{v}= %
,
$ $v_{+}=\underline{\infty}$, $v_{-}=\underline{0}$, and we consider these two points,
$\underline{\infty}$ and $\underline{0}$, to be \underline{both} the mothers and fathers
of $v=\underline{1}$. \\
If $v=(x,y)$ has continued fraction expansion (\ref{eq:24}), then 
\begin{equation}
				%
				\label{eq:31}
\end{equation}
Note that if $F(v)=v_{\delta}$, there is a unique $m=m(v)\ge 1$, such that 
\begin{equation}
				v_{\delta} = F(v) = M^{1+m}(v)
				\label{eq:32}
\end{equation}
i.e. every father is an ``m-th great grandmother'', and more precisely we have 
\begin{equation}
				v=  
				\left\{%
				$
\end{minipage}
				\caption{ Concavity and convexity of $\delta(v)$ }
				\label{fig:2}
\end{figure} \ \\
The \underline{direct descendants} of $v\in\Q^{+}$ are its offsprings 
$M^{-1}\left( v \right)= \left\{ v+v_{+},v+v_{-} \right\}$, as well as the elements of
$F^{-1}(v)$, and they form the \underline{Fin around $v$}, with its positive and negative
parts: 
\begin{equation}
				%
 \\
				\caption{Fin around $v$}
				\label{fig:3}
\end{figure} 
\section{The collection of $\lideal$-sets}
The binary tree structure on $\Q^{+}$ gives a partial order $\lideal$ on $\Q^{+}$, 
\begin{equation}
				v^{\prime}\lideal v \quad \text{iff}\quad v^{\prime} = M^n(v) \quad
				\text{for some $n\ge 0$.}
				\label{eq:35}
\end{equation}
\begin{definition}
				A $\lideal$-set $c$ is a finite subset of $\Q^{+}$, such that 
				\[ v\in c, \quad v^{\prime}\lideal v \Longrightarrow v^{\prime} \in c \]
				\label{def:36}
or equivalently, $M(c)\subseteq c$ and $c$ is a \underline{finite subtree} of $\Q^{+}$.
\end{definition}  
We let $\mathcal{C}$ denote the collection of all $\lideal$-sets, it is a lattice: 
\begin{equation}
				c,c^{\prime}\in \mathcal{C} \Longrightarrow c\cap c^{\prime}, \quad c\cup
				c^{\prime}\in \mathcal{C}
				\label{eq:37}
\end{equation}
\begin{equation}
				\mathcal{C}=\coprod_{m\ge 1}\mathcal{C}_{m}\quad , \quad \mathcal{C}_{m}=\left\{
				c\in \mathcal{C}, \# c= m-1 \right\}
				\label{eq:38}
\end{equation}
For $v\in \Q^{+}$ we have the $\lideal$-set $c_{v}$ consisting of the path from the root
$\Unit=(1,1)$ to $v$, 
\begin{equation}
				c_v = \left\{ v^\prime\in \Q^{+}, v^{\prime}\lidealeq v \right\} = \left\{
								M^{n}(v)
				\right\}_{n\ge 0}
				\label{eq:39}
\end{equation}
We have
\begin{equation}
				\begin{array}[H]{ll}
				c_{v_1}\cap c_{v_2} = c_{v_1\Lambda v_2} \\\\
				\Lambda:\Q^{+}\times \Q^{+} \to \Q^{+} \quad \text{associative, commutative and }
				\\\\
				v_{1}\lideal v_{2} \Longleftrightarrow v_{1}\Lambda v_{2}= v_{1} \\\\
				v_{1}< v_{2} \Longleftrightarrow v=v_{1}\Lambda v_{2} \quad \text{satisfies} \quad
				v+v_{+}\lideal v_{2} \text{ or } v+v_{-}\lideal v_1.
				\end{array}
				\label{eq:40}
\end{equation}
Every $\lideal$-set $c$ is determined by its set of $\lideal$-maximal elements
$c^{\max}$: 
\begin{equation}
				c = \bigcup\limits_{v\in c^{\max}}c_{v}
				\label{eq:41}
\end{equation}
For a $\lideal$-set $c\in \mathcal{C}_{m}$, we write its elements in increasing $\le$
order 
\begin{equation}
				c= \left\{ c_{m-1}>\cdots > c_2 > c_{1} \right\}\quad , \quad m=1+\# c,
				\label{eq:42}
\end{equation} 
and we put 
\begin{equation}
				\begin{array}[H]{lll}
				c_{m}=\underline{\infty}\quad , \quad c_{0}=\underline{0} \\\\
				\partial c = \left\{ [c_{i-1},c_i] \right\}_{i=1}^{m}
				\end{array}
				\label{eq:43}
\end{equation}
The collection of edges $\partial c=\left\{ \left[ c_{i-1},c_{i} \right] \right\}$ forms a
\underline{polygonal path} in the Farey graph going from $\underline{0}$ to
$\underline{\infty}$. Conversely  every path in the Farey graph $\left\{ \left[
								c_{i-1},c_{i}
\right] \right\}$, $c_{0}=\underline{0}$, $c_m=\underline{\infty}$, $\det
\begin{pmatrix}c_{i-1} \\ c_{i}\end{pmatrix}=1$, forms a $\lideal$-set $c=\left\{
				c_{i}
\right\}_{i=1}^{m-1}\in \mathcal{C}_{m} $. \\ 
Thus the finite subtrees of $\underline{\Q}^{+}$ are
precisely the paths in $\fg$ from $\underline{0}$ to $\underline{\infty}$. 
\section{Structure of $\lideal$-sets}
For $c=\left\{
c_{i} \right\}\in \mathcal{C}$ we have 
\begin{equation}
				\begin{array}[H]{lll}
								c^{\max} &= \left\{ c_i,c_{i+1}\lideal c_i, c_{i-1}\lideal c_i \right\} \\\\
								&= \left\{ c_{i},c_{i}=c_{i+1}+c_{i-1} \right\} \\\\
								&= \left\{ v\in c,\; v+v_{+}\not\in c \; \text{and} \; v+v_{-}\not\in c
								\right\} \text{ the leaves of the tree $c$.}
				\end{array}
				\label{eq:45}
\end{equation}
We also define the \underline{local $\lideal$-minima} 
\begin{equation}
				\begin{array}[H]{lll}
								\begin{array}[H]{lll}
												\Phi c &= \left\{ c_{j}, c_{j}\lideal c_{j+1} \; \text{and}\;
												c_{j}\lideal c_{j-1} \right\} \\\\
												&= \left\{ v\in c, v+v_{+}\in c \; \text{and} \; v+v_{-}\in c \right\} 
								\end{array}
				\end{array}
				\label{eq:46}
\end{equation}
and we put
\begin{equation}
				c^{\min} = \Phi c \amalg \left\{ \underline{0},\underline{\infty} \right\}.
				\label{eq:47}
\end{equation}
These sets are intertwined: 
\begin{equation}
				\begin{array}[H]{lll}
								c^{\max} =  \left\{c_{i_{\ell}}>\cdots > c_{i_{1}}>c_{i_{0}}  \right\}  \\\\
				c^{\min} =  \left\{ \underline{\infty}>c_{j_{\ell}}> \cdots > c_{j_{1}}>
				\underline{0} \right\} \\\\
				\text{with} \quad m>i_{\ell}>j_{\ell}> \cdots > i_{1}>j_{1}>i_{0}> 0
				\end{array}
				\label{eq:48}
\end{equation}
\ \\
\begin{remark}
				There are still more perspectives on $\lideal$-sets. \\
				With $c\in\mathcal{C}_{m}$ we have the \underline{triangulated polygone} 
				\[\Delta (c):= \bigcup\limits_{v\in c}\Delta_{v}\]
				and the associated \underline{Friez Pattern} 
				\[ 
								\begin{array}[H]{lll}
								f(c):= \left\{ f(c)_{0}, f(c)_{1}, \cdots , f(c)_{m} \right\}, \quad
								f(c)_{j}=\# \left\{ v\in c, c_{j}\in \Delta_{v} \right\}	\\\\
								c^{\max}=\left\{ c_{j},f(c)_{j}=1 \right\}
\end{array} \]
				\label{remark:44}
\end{remark}
Clearly $f(c)$ determines $c$. We have the associated \underline{bipartite graph} 
\[
				\begin{array}[H]{lll}
				 c \amalg \left\{ \underline{0},\underline{\infty} \right\}
				 \xleftarrow{\;\;\pi_{0}\;\;} \mathbb{B}(c)\xrightarrow{\;\;\pi_{1}\;\;}c
								\\\\ 
								\mathbb{B}(c):= \left\{ (v^{\prime},v), v^{\prime}\in \Delta_{v},v\in c \right\}
								\\\\
								\# \pi_{0}^{-1}(c_{j}) = f(c)_{j}, \quad j=0,\cdots,m \; ;  \quad \#
								\pi_{1}^{-1}(c_{j}) \equiv 3, \;\; j=1,\cdots,m-1.
\end{array}
\]
\ \\
\begin{remark}
				The monoid $\SL_{2}(\N)$ has two commuting \underline{involutions.} \\
				One is the
				automorphism (outer in $\SL_{2}(\Z)$, inner in $\GL_{2}(\Z)$), 
				\begin{equation*}
								\begin{array}[H]{lll}
												g = 
												\begin{pmatrix}
																x_{-} & y_{-} \\
																x_{+} & y_{+} 
												\end{pmatrix}	
								\longmapsto g^{\star} 
:= 
\begin{pmatrix}
				0 & 1 \\
				1 & 0
\end{pmatrix} 
g
\begin{pmatrix}
				0 & 1 \\
				1 & 0
\end{pmatrix} 
= 
\begin{pmatrix}
				y_{+} & x_{+} \\
				y_{-} & x_{-}
\end{pmatrix} 
\\\\
\left( g_{\pm} \right)^{\star} = g_{\mp}\quad , \quad \left( g_{1}\cdot g_{2}
\right)^{\star} = g_{1}^{\star}\cdot g_{2}^{\star}\quad , \quad g^{\star\star}=g. 
								\end{array}
				\end{equation*}
				\label{remark:2}
\end{remark}
The other is the \underline{anti}-automorphism
\begin{equation*}
				\begin{array}[H]{ll}
								g= 
								\begin{pmatrix}
												x_{-} & y_{-} \\
												x_{+} & y_{+}
								\end{pmatrix} \longmapsto g^{t}:= 
								\begin{pmatrix}
												x_{-} & x_{+} \\
												y_{-} & y_{+}
								\end{pmatrix} \\\\
								\left( g_{\pm} \right)^{t} = g_{\mp} \quad , \quad \left( g_1 \cdot g_2
								\right)^{t} = g_{2}^{t}\cdot g_{1}^{t}\quad , \quad g^{tt}=g.
				\end{array}
\end{equation*}
In terms of the identification $\SL_{2}(\N)\equiv \Q^{+}$ these read: 
\begin{equation*}
				\begin{array}[H]{ll}
								v^{\star}= \left( y/x \right)^{\star} = \left(
								\frac{y_{+}+y_{-}}{x_{+}+x_{-}} \right)^{\star} = \frac{
								x_{-}+x_{+}}{y_{-}+y_{+}} = x/y = v^{-1} \\\\
								v^{t} = \left( \frac{y_{+}+y_{-}}{x_{+}+x_{-}} \right)^{t} =
								\frac{y_{+}+x_{+}}{y_{-}+x_{-}} = \frac{\;|v_{+}|_{1}\;}{|v_{-}|_{1}}. 
				\end{array}
\end{equation*}
\subsection{Creation and annihilation operators}
For $c=\left\{ c_{i} \right\}_{i=1}^{m-1}\in\mathcal{C}_{m}$, we have $c\cup \left\{
				c_{i}+c_{i-1}
\right\}\in \mathcal{C}_{m+1}$, $i=1,\cdots, m$, and we obtain the \underline{creation
operator}
\begin{equation}
				\begin{array}[H]{lll}
				d^{\star}:\Z\mathcal{C}_{m}\longrightarrow \Z\mathcal{C}_{m+1}\\\\
				d^{\star}[c]:= \sum\limits_{i=1}^{m}\left[ c\cup \left\{ c_{i}+c_{i-1} \right\}
				\right].
				\end{array}
				\label{eq:49}
\end{equation}
Similarly, for $c_{j}\in c^{\max}$, we have $c\setminus \left\{ c_{j} \right\}\in
\mathcal{C}_{m-1}$, and we obtain the \underline{annihilation} \underline{operator} 
\begin{equation}
				\begin{array}[H]{lll}
				d:\Z\mathcal{C}_m \longrightarrow \Z\mathcal{C}_{m-1} \\\\
				d[c] := \sum\limits_{c_{j}\in c^{\max}} \left[ c\setminus \left\{ c_{j}
				\right\} \right] 
				\end{array}
				\label{eq:50}
\end{equation}
Since the operation of adding a mediant, and of removing a (different) maximal point
commute, we see that the \underline{Number operator}
\begin{equation}
				N=d\circ d^{\star}-d^{\star}\circ d : \Z \mathcal{C}_{m}\longrightarrow
				\Z\mathcal{C}_{m}
				\label{eq:51}
\end{equation}
is diagonalizable in the basis of $\lideal$-sets, and we have 
\begin{equation}
				N[c] = \left( m-\# c^{\max} \right)\cdot [c], \quad c\in \mathcal{C}_m .
				\label{eq:52}
\end{equation}
\ \\
\begin{remark}
				\underline{Friez indices}: For $c\in\mathcal{C}$, $c_{i}=v\in c$, we have: 
				\begin{equation*}
								\begin{array}[H]{lll}
												c_{i-1}\in f_{v}^{-}\amalg \left\{ v_{-} \right\}=\left\{
																v_{-}+n\cdot v, n\ge 0
												\right\},\quad c_{i-1}=v_{-}+n\cdot v, \quad n=n(c)_{i
												}^{-}\ge 0, \\\\
												c_{i+1}\in f_{v}^{+}\amalg \left\{ v_{+} \right\}= \left\{
																v_{+}+n\cdot v, n\ge 0
												\right\}, \quad c_{i+1} = v_{+}+ n\cdot v, \quad
												n=n(c)_{i}^{+}\ge 0. \\\\
												c_{i}\in c^{\max} \Longleftrightarrow
												n(c)_{i}^{-}=n(c)_{i}^{+} = 0 \\\\
												c_{i}\in \Phi c \Longleftrightarrow n(c)_{i}^{-}> 0 \;\;
												\text{and} \;\; n(c)_{i}^{+}>0 \\\\
												f(c)_{i}=1+n(c)_{i}^{+}+n(c)_{i}^{-}, \quad i=1,\cdots , m-1.
								\end{array}
				\end{equation*}
				\label{remark:53}
Note that if $n=n(c)_{i}^{-}>0$, then there exists %\vspace{.1cm}\\
\begin{equation}
				i_{0}<i_{1}<\cdots < i_{n}\equiv i-1, \quad \text{with} \quad c_{i_{k}} =
				v_{-}+k\cdot v, \quad k=0,\cdots n;
				\label{eq:54}
\end{equation}\ %\vspace{.1cm}\\
Similarly, if $m=n(c)_{i}^{+}>0$, then there exists 
% \ \vspace{.1cm}\\
\begin{equation}
				j_{0}>j_{1}> \cdots > j_{m}\equiv i+1, \quad \text{with}\quad c_{j_{k}} =
				v_{+}+k\cdot v, \quad k=0,\cdots , m\; ;
				\label{eq:55}
\end{equation} 
\end{remark} 
\subsection{The operad structure}
For $c=\left\{ c_{i} \right\}_{i=1}^{m-1}\in\mathcal{C}_{m}$, so
$\begin{pmatrix}c_{i-1}\\ c_{i}\end{pmatrix}\in \SL_{2}(\N)$, and for 
$b^{(i)}=\left\{ b_{j}^{(i)} \right\}_{j=1}^{n_{i}-1}\in\mathcal{C}_{n_{i}}$ $i=1,\cdots,
m$ (as usual $c_{m}=b_{n_{i}}^{(i)}=\underline{\infty}$, $c_{0}=
b_{0}^{(i)}=\underline{0}$), we obtain the path $I_{i}=\left\{ b_{j}^{(i)}
\begin{pmatrix}c_{i-1} \\ c_{i}\end{pmatrix} \right\}_{j=0}^{n_{i}}$ from
$\underline{0} \begin{pmatrix}c_{i-1}\\ c_{i}\end{pmatrix} = c_{i-1}$ to
$\underline{\infty}\begin{pmatrix} c_{i-1}\\ c_{i}\end{pmatrix}=c_{i}$, and together
$\bigcup\limits_{i}I_{i}$ gives a new path from $\underline{0}$ to $\underline{\infty}$
in $\fg$, denoted by $c\circ b$, and we obtain, 
\begin{proposition}
				The set $\mathcal{C}=\coprod\limits_{m\ge 1}\mathcal{C}_{m}$ is an
				\underline{operad} via 
				\begin{equation*}
								\begin{array}[H]{l}
								\begin{array}[H]{ccccccrl}
								\mathcal{C}_{m} &\times &\mathcal{C}_{n_1} &\times &\cdots &\times
								&\mathcal{C}_{n_{m}} & \xrightarrow{\qquad}\mathcal{C}_{n_1+\ldots+
								n_{m}}   \\\\
								c &, &b^{(1)} &, &\cdots &, &b^{(m)} & \xmapsto{\qquad} c\circ b 
								\end{array}
								\\\\
								c \circ b:= \left\{ b_{j}^{(i)} 
\begin{pmatrix}
				c_{i-1} \\ c_{i}
\end{pmatrix}\right\}\;\; 1\le i \le m, \;\; 1\le j\le n_{i} \\\\
\left( c\circ b \right)\circ a = c\circ \left( b\circ a \right)\quad , \quad \left\{
				\phi 
\right\}\circ c = c = c\circ \left\{ \phi \right\}^{m}. 
								\end{array}
				\end{equation*}
				\label{prop:56}
\end{proposition}
The unit is the empty $\lideal$-set $\phi$, $\mathcal{C}_{1}=\left\{ \phi
\right\}$, $\partial \phi =\left\{ [0,\infty] \right\} $; the root
$\underline{1}=(1,1)$
satisfy $\partial\left\{ \underline{1}\right\}=\left\{
[\underline{0},\underline{1}],[\underline{1},\underline{\infty}] \right\}$, $\Phi\left\{
				\underline{1}\right\}=\phi
$;
and we put $\Phi(\phi)=\phi$. \\
\section{Coronas}
We shall identify the integers $\Z$ with
$\lideal$-set via 
\begin{equation}
				\Z\ni m \xleftrightarrow{\qquad}\nu_{m}=
				\left\{  
				\begin{array}[H]{l}
								\left\{ (1,1),(1,2),\cdots, (1,m+1) \right\}, m\ge 0 \\\\
								\left\{ (1,1) \right\}\quad m=0 \\\\
								\left\{ (1,1),(2,1),\cdots, (|m|+1,1) \right\},\;\; m\le0

				\end{array}
\right\}\in \mathcal{C}_{|m|+2}
				\label{eq:57}
\end{equation}
Note that the $\lideal$-set $\nu_{m}$ has a unique $\lideal$-maximal point, and the path
$\partial \nu_{m}$ consists of a straight line from $\underline{0}$ to the maximal point,
followed by a straight line from the maximal point to $\underline{\infty}$; this properly
characterizes the $\lideal$-sets $\left\{ \nu_{m} \right\}, \; m\in \Z$. 
% I am here
Thus for a $\lideal$-set $c\in\mathcal{C}_{m}$, and vector
$\nu=\sum\limits_{i=1}^{m}n_{i}\cdot \left[ c_{i-1},c_{i} \right]\in \Z\partial c$, we
obtain the $\lideal$-set 
\begin{equation}
				c\circ \nu:=c\circ \left\{ \nu_{n_i} \right\}\in \mathcal{C}_{2m+|v|} \quad , \quad
				|\nu|=\sum\limits_{i=1}^{m} |n_{i}|.
				\label{eq:58}
\end{equation}
Explicitly, the $\lideal$-set $c\circ \nu$ is obtained by replacing $\left\{
				c_{i}>c_{i-1}
\right\}$ in $c$ by
\begin{equation}
				\begin{array}[H]{lll}
								\left\{ c_{i}> (n+1)\cdot c_{i}+c_{i-1}> \cdots > c_{i}+c_{i-1}>c_{i-1} \right\}, \quad
				n_{i}=n\ge 0, \\\\
\left\{ c_{i}>c_{i}+c_{i-1}> c_{i-1} \right\}\quad , \quad n_{i}=0, \\\\
				\left\{ c_{i}> c_{i}+c_{i-1}>\cdots > c_{i}+(|n|+1)\cdot c_{i-1}> c_{i-1}
				\right\}, \quad n_{i}=n\le 0. 
				\end{array}
				\label{eq:59}
\end{equation}
We have 
\begin{equation}
				\Phi(c\circ \nu) = c.
				\label{eq:60}
\end{equation}
The $\lideal$-set $c\circ \nu$ is a $\star$-set according to the 
\begin{definition}
				$A$ $\star$-set $c\in \mathcal{C}_{m}^{\star}$ is a $\lideal$-set $c\in
				\mathcal{C}_{m}$ such that for any consecutive edges of $\partial c$,
				$c_{i+1}> c_{i}> c_{i-1}$, we have: 
				\begin{equation*}
								\begin{array}[H]{lll}
												\text{either} & c_{i}\in c^{\max}\Longleftrightarrow c_{i}= c_{i+1}+c_{i-1} \\\\
								\text{or}\quad & c_{i}\in c^{\min} \Longleftrightarrow c_{i}\lideal
								c_{i+1} \quad \text{and}\quad c_{i}\lideal c_{i-1} 
								\end{array}
				\end{equation*}
				or they form a straight line: $c_{i}-c_{i-1}= c_{i+1}-c_{i}\Longleftrightarrow
				c_{i}=\frac{1}{2}\left( c_{i+1}+c_{i-1} \right)$. 
				\label{eq:61}
\end{definition}
We write $\mathcal{C}^{\star}=\coprod\limits_{m\ge 1}\mathcal{C}_{m}^{\star}$ for the collection of
$\star$-sets. 
\begin{remark}
				For $c\in\mathcal{C}^{\star}$, and for $c_{i'}>c_{i} $ two consecutive points of
				$c^{\min}$, let $c_{j}$ be the unique point of $c^{\max}$ between them, $i'>j>i$,
				we have either a positive or a negative fin between them: there exists $m>n\ge 0$
				such that either 
				\begin{equation}
								\begin{array}[H]{ll}
												\left( f_{c_{i}}^{+} \right): \;
												c\cap \left[ c_{i}, c_{i'} \right] \equiv  
																\left\{\begin{array}[H]{rl}
																c_{i'}\equiv (c_{i})_{+}+n\cdot c_{i} &>
																(c_{i})_{+}+(n+1)c_{i}  
																>\cdots \\
																\cdots &> c_{j}=c_{i+1}= (c_{i})_{+}+mc_{i}>
												c_{i} 
								\end{array}\right\}
				\\\\
				\text{(note that $c_{i'}\lideal c_{i}$; unless $n=0$ and $c_{i}\lideal
				c_{i'}$); put $\lambda_{c}\left( \left[ c_{i}, c_{i'} \right] \right)=n-m+1$; or}
				\\\\
				\left( f_{c_{i'}}^{-} \right): c\cap \left[ c_{i},c_{i'} \right]\equiv
				\left\{ 
								\begin{array}[H]{rl}
								c_{i'}>c_{j}=c_{i'-1}= (c_{i'})_{-}+mc_{i'} &>\cdots \\ \cdots &> (c_{i'})_{-}+
				nc_{i'} =c_{i}
								\end{array}
\right\} 
								\end{array}
								\label{eq:63}
				\end{equation}
				\label{remark:62}
\end{remark}
(note that $c_{i}\lideal c_{i'}$; unless $n=0$ and $c_{i'}\lideal c_{i}$); put
$\lambda_{c}\left( \left[ c_{i},c_{i'} \right] \right)=m-n-1$. Thus in any case
$[c_{i}, c_{i'}]$ is an edge of the Farey graph, so that $\partial c^{\min}$ is again a
path from $\underline{0}$ to $\underline{\infty}$ in $\fg$, and $\Phi c\in
\mathcal{C}$ is again a $\lideal$-set. Thus any $\star$-set $c$ can be written (uniquely)
as 
\begin{equation}
				c= \Phi c\circ \lambda_{c}. 
				\label{eq:64}
\end{equation}
We obtain the fibration  
\begin{equation}
				\begin{array}[H]{lll}
								\Phi: \mathcal{C}^{\star}\xtworightarrow{\quad} \mathcal{C} \\\\
								c\mapsto \Phi c= c^{\min}\setminus \left\{
								\underline{0},\underline{\infty} \right\} \\\\
								\Z \partial \overline{c}\equiv \Phi^{-1}(\overline{c})
				\end{array}
				\label{eq:65}
\end{equation}
We get the pull-back diagram 
\begin{equation}
				\begin{array}[H]{lllll}
								\mathcal{C}^{\star} \xrightarrow{\qquad\Phi\qquad} & \mathcal{C} & \ &  \\\\
								\text{\protect\rotatebox{90}{$\subseteq$}} \;\;  &
								\text{\protect\rotatebox{90}{$\subseteq$}} &  \ &  \\\\
				\Phi^{-1}(\mathcal{C}^{\star})\xrightarrow{\qquad\Phi\qquad}
				&\mathcal{C}^{\star}& \xrightarrow{\qquad\Phi\qquad}  & \mathcal{C}    & \\\\
				\text{\protect\rotatebox{90}{$\subseteq$}} &
				\text{\protect\rotatebox{90}{$\subseteq$}}   & &
				\text{\protect\rotatebox{90}{$\subseteq$}}  \\\\
				\Phi^{-2}(\mathcal{C}^{\star})\xrightarrow{\qquad\Phi\qquad} & 
				\Phi^{-1}(\mathcal{C}^{\star})& \xrightarrow{\qquad\Phi\qquad} & 
				\mathcal{C}^{\star} \xrightarrow{\qquad\Phi\qquad}  \mathcal{C}  & \\\\
				\vdots  
				\end{array}
				\label{eq:66}
\end{equation}
\begin{definition}
				The set of \underline{Coronas} is 
				\begin{equation*}
								\begin{array}[H]{lll}
												\cor:= \bigcap\limits_{n\ge 0} \Phi^{-n}(\mathcal{C}^{\star}) =
												\coprod\limits_{m\ge 1}\cor_{m} \\\\
												\cor_{m}=\left\{ c\in \cor, \# c= m-1 \right\}
								\end{array}
				\end{equation*}
				\label{definition:1.67}
\end{definition}
We get the fibration 
\begin{equation}
				\begin{array}[H]{cc}
				\Phi:&\cor\xrightarrow{\qquad}\cor \\\\
				\ & c\xmapsto{\qquad} \Phi c \\\\
				\Z \partial \overline{c} &\equiv \Phi^{-1}(\overline{c})

				\end{array}
				\label{eq:68}
\end{equation}
Given $c\in \cor_{m}$ and any vector $\nu=\sum\limits_{i=1}^{m}n_{i} \left[
				c_{i-1},c_i
\right] \in \Z\partial c­$, we get $c\circ \nu\in \cor_{2m+|\nu|},
|\nu|=\sum\limits_{i=1}^{m}|n_{i}|$, and $\Phi(c\circ \nu)=c$. Conversely, any $c\in
\cor_{m}$ can be written uniquely as 
\begin{equation}
				c=\Phi(c)\circ\lambda_{c}, \qquad \text{with}\quad \Phi(c)\in
				\cor_{\frac{1}{2}\left( m-|\lambda_{c}| \right)}, \quad \lambda_c \in \Z\partial
				\Phi c. 
				\label{eq:69}
\end{equation}
\section{Structure of coronas}
We give next a constructive approach to coronas based on the \vspace{.1cm}\\
\underline{Inductive Principle}: For $c\in \cor_{m}$, $m>1$, there exists 
\begin{equation*}
				c_{j}\in c^{\max} \;\; \text{such that} \;\; c\setminus \left\{ c_{j} \right\}\in
				\cor_{m-1}.
\end{equation*}
Indeed writting $c=\Phi c\circ \lambda_{c}$, $\lambda_{c}=\sum n_{i}\left[
				\overline{c}_{i-1},\overline{c}_{i}
\right]$, $\Phi c=\overline{c}=\left\{ \overline{c}_{i} \right\}$, if $n_{i_{0}}\not = 0$
we can take $c_{j}$ to be the $\lideal$-maximal element in $c\cap \left[
\overline{c}_{i_{0}-1},\overline{c}_{i_0} \right]$, $\Phi(c\setminus \left\{ c_{j}
\right\})\equiv \Phi c$. Otherwise, $\lambda_{c}\equiv 0$, $\overline{c}=\Phi c \in
\cor_{\frac{m}{2}}$, by induction there is $\overline{c}_{i}\in
\overline{c}^{\max}$ such that $\overline{c}\setminus \left\{ \overline{c}_{i}
\right\}\in\cor_{\frac{m}{2}-1}$, and we can take $c_{j}$ to be either
$\overline{c}_{i+1}+\overline{c}_{i}$ or $\overline{c}_{i}+\overline{c}_{i-1}$.
\vspace{.4cm}\\ 
Thus every
$c\in \cor_{m}$ is obtained from the empty corona $\phi\in\cor_{1}$ by adding
one point at a time, and the set $\cor=\coprod\limits_{m\ge 1}\cor_{m}$ forms
the vertices of a connected rooted graph with edges 
\begin{equation}
				\cor^{1} \equiv \left\{ (c,c^{\prime})\in\cor\times\cor , c\subseteq
				c^{\prime}, \# c^{\prime}=\# c+1 \right\}
				\label{eq:71}
\end{equation}
cf. figure 4.  %\vspace{2cm}\\
\begin{figure}[H]
				\centering

				\caption{First six levels of Coronas}
				\label{fig:4}
\end{figure} \  \vspace{-1cm}
\begin{figure}[H]
				%
 \ \\ \
\caption{Level seven of the corona tree}
\label{fig:5}
\end{figure} \ \vspace{1cm} \\ \ 
\subsection{Creation and annihilation operators}
For $c\in\cor_{m}$ define the set of \underline{closed points} of $c$,
$\cl(c)\subseteq c^{\max}$, by declaring all $c_{j}\in c^{\max}$ to be closed accept
when $c_{j+1},c_{j-1}\in c^{\min}$, and we have 
\begin{equation}
				\begin{array}[H]{llll}
								\text{either} & c_{j+1}\lideal c_{j-1}\lideal c_{j-2} & \text{and} &
								c_{j-1} \neq \frac{1}{2}\left( c_{j+1}+c_{j-2} \right) \\\\
								\text{or} & c_{j-1}\lideal c_{j+1}\lideal c_{j+2} & \text{and} &
								c_{j+1}\neq \frac{1}{2}\left( c_{j+2}+c_{j-1} \right)
				\end{array}
				\label{eq:72}
\end{equation}
Thus, 
\begin{equation}
				\begin{array}[H]{ll}
				\cl(c) = \left\{ c_{j}\in c^{\max}, c\setminus\left\{ c_{j} \right\}\in
				\cor_{m-1} \right\} \\\\
				\# \cl(c) = 1+\# \Phi c -h^\circ(c), \quad 0\le h^\circ(c) \le \# \Phi c.
\end{array}
\label{eq73}
\end{equation}
We obtain the \underline{annihilation operator}
\begin{equation}
				\begin{array}[h]{ll}
				d:\Z\cor_{m}\xrightarrow{\qquad}
				\Z\cor_{m-1} \\\\
				d\left[ c \right] 
				% =\sum\limits_{c_{j}\in \cl(c)} 
				% d:\Z\cor_{m}\xrightarrow{\quad} \Z\cor_{m-1}
				% d[c] 
				= \sum\limits_{c_{j}\in\cl(c)} \left[ c\setminus \left\{ c_{j}
								\right\} \right]
				\end{array}
				\label{eq:74b}
\end{equation}
% \begin{equation}
% 				\begin{array}[H]{l}
% 				\cl(c) = \left\{ c_{j}\in c^{\max}, c\setminus \left\{ c_{j}\right\} \in\cor
% 				\right\} \\\\
% 				\#\cl(c) = 1+\#\Phi c - h^\circ(c), \quad 0\le h^\circ(c)\le \# \Phi c. 
% 				\end{array}
% 				\label{eq:73}
% \end{equation}
Similarly, for $c\in \cor_{m}$ define the set of \underline{open edges} of $c$
\begin{equation}
				\begin{array}[H]{lll}
				\op(c)=\left\{ \left[ c_{i-1},c_{i} \right]\in\partial c, c\cup \left\{
												c_{i}+c_{i-1}
				\right\}\in\cor_{m+1} \right\} \\\\
				\#\op(c)=2+2\cdot \#\Phi c - h^{1}(c)\quad, \quad 0\le h^{1}(c) \le
				\#\Phi c.
				\end{array}
				\label{eq:75}
\end{equation}
Indeed, for $c_{i^{\prime}}>c_{i}$ consecutive points of $c^{\min}$, $c_{j}\in
c^{\max}$ the $\lideal$-maximal point between them, cf. (\ref{eq:63}), the edge $\left[
c_{i},c_{j} \right]\equiv \left[ c_{i},c_{i+1} \right]$ (resp. $\left[
c_{j},c_{i^{\prime}} \right]\equiv \left[ c_{i^{\prime}-1},c_{i^{\prime}} \right]$) is
always open in case of $f_{c_{i}}^{+}$ (resp. $f^{-}_{c_{i^{\prime}}}$); the only other
possibly open edge in $c\cap \left[ c_{i},c_{i^{\prime}} \right]$ is the edge
$[c_{i^{\prime}-1},c_{i^{\prime}}]$ (resp. $[c_{i},c_{i+1}]$), and this edge is open iff $\left[
								c_{i},c_{i^{\prime}}
				\right]$ is open in $\Phi c$. \\ 
We obtain the \underline{creation operator} 
				\begin{equation}
								\begin{array}[H]{l}
								d^{\star}:\Z\cor_{m}\xrightarrow{\quad}\Z\cor_{m+1} \\\\
								d^{\star}[c] = \lsum\limits_{\left[ c_{i-1},c_{i} \right]\in\op(c)} \left[
								c\cup \left\{ c_{i}+c_{i-1} \right\} \right]
								\end{array}
								\label{eq:76}
				\end{equation}
				Since the operations of adding a mediant and that of removing a (different)
				maximal point commute, we see that the \underline{Number operator} 
				\begin{equation}
								N=d\circ d^{\star} - d^{\star}\circ d: \Z\cor_{m}\xrightarrow{\qquad}
								\Z\cor_{m}
								\label{eq:77}
				\end{equation}
is diagonalizable in the basis of coronas with eigenvalues 
\begin{equation}
				\begin{array}[H]{c}
								N[c] = e_{c}\cdot [c] \\\\
								e_{c} = \#\op(c)-\#\cl(c) = 1 +\#\Phi c +
								h^{0}(c)-h^{1}(c), \qquad 1\le e_{c}\le 1+2\cdot \#\Phi c \\\\
								\text{and when } h^{0}(c)=h^{1}(c):e_{c}=1+\#\Phi c = \#
								c^{\max}.
				\end{array}
				\label{eq:78}
\end{equation} 
\subsection{The d.n.a. of a corona}
For $c\in \cor_{m}$ define its \underline{height} to be 
\begin{equation}
				\text{ht}(c) = \min \left\{ n, \Phi^{n}c = \phi \right\}
				\label{eq:7.9}
\end{equation}
For $n=1,2,\cdots, \text{ht}(c)$ we have $\Phi^{n}c\in \cor_{\ell_{n}}$, and 
\begin{equation}
				\Phi^{n-1}c = \Phi^{n}c\circ \lambda_{c}^{n} \quad , \quad
				\lambda_{c}^{n}\in\Z\partial\Phi^{n}c\cong \Z^{\ell_{n}}. 
				\label{eq:7.10}
\end{equation}
So that we have 
\begin{equation}
				c=\phi \circ \lambda_{c}^{\text{ht}(c)}\circ \cdots \lambda_{c}^{n}\circ
				\cdots \lambda_{c}^{1}
				\label{eq:7.11}
\end{equation}
Thus $c$ is determined by the $\Z$-valued vectors 
\begin{equation}
				\lambda_{c}^{n} = \left( \lambda_{1}^{n},\lambda_{2}^{n},\ldots ,
				\lambda_{\ell_n}^{n} \right)\in \Z^{\ell_{n}}, \;\; \lambda_{j}^{n}\in \Z,
				\label{eq:7.12}
\end{equation}	
				of length $\ell_n=2\cdot \ell_{n+1}+|\lambda_{c}^{n+1}|$,
				 $\left|\lambda_{c}^{n+1}\right|=\sum\limits_{j}|\lambda_{j}^{n+1}|$; $\ell_{\text{ht}(c)}=1$.
				\\
				We refer to $\lambda_c^{\text{ht}(c)}\in\Z, \ldots ,
				\lambda_{c}^{n}\in \Z^{\ell_n}, \ldots , \lambda_{c}^{1}\in
				\Z^{\ell_1}$ as the \underline{d.n.a. of $c$.} We have 
				\begin{equation}
								\begin{array}[H]{ll}
								m &= \# c+1=2^{\text{ht}(c)}+2^{\text{ht}(c)-1}\cdot
								|\lambda_c^{\text{ht}(c)}|+ \cdots 2^{n}\cdot |\lambda_{c}^{n+1}|+\cdots
								|\lambda_{c}^{1}| \\\\ 
								&= "\ell_0" = 2\cdot
								\ell_1+|\lambda_{c}^{1}|
								\end{array}
								\label{eq:7.13}
				\end{equation}
\section{The main examples}
Beside the real total order $\le$, and the tree partial order $\lideal$, we shall use the
following two partial orders on $\underline{\Q}^{+}$. We have the \underline{pointwise
order} $\prec$ :
\begin{equation}
				\begin{array}[H]{l}
				v\prec v^{\prime} \Longleftrightarrow v^\prime-v\in \N\times \N \\
				\text{or} \\
				(x,y)\prec (x^{\prime},y^{\prime}) \Longleftrightarrow x\le x^{\prime} \text{ and
				} y\le y^{\prime}
				\end{array}
				\label{eq:79}
\end{equation}
We have the \underline{fundamental order} $\llcurly$ :
\begin{equation}
				v\llcurly v^{\prime} \Longleftrightarrow v_{+}\prec v_{+}^{\prime} \quad
				\text{and} \quad v_{-}\prec v_{-}^{\prime}
				\label{eq:80}
\end{equation}
Note that we have the strict implications 
\begin{equation}
				v\lideal v^{\prime} \Longrightarrow v\llcurly v^{\prime}\Longrightarrow v\prec
				v^{\prime}
				\label{eq:81}
\end{equation}
Let $||:\underline{\Q}^{+}\to \R$ be any map that is \underline{fundamentally monotone}: 
\begin{equation}
				v\llcurly v^{\prime} \Longrightarrow |v|\le |v^{\prime}|
				\label{eq:82}
\end{equation}
For $R\in\R$ put 
\begin{equation}
				c\left( \absarg\le R \right):= \left\{ v\in\underline{\Q}^{+}, |v|\le R \right\}
				\label{eq:83}
\end{equation}
\begin{theorem}
				Assuming $c\left( \absarg\le R \right)$ is finite, then it is a corona.
				\label{thm:84}
\end{theorem}
Note that $v\lideal v^{\prime}$ imply $|v|\le |v^{\prime}|$, (\ref{eq:81}-\ref{eq:82}), so
that $c(\absarg \le R)$ is a $\lideal$-set. \\ 
We begin the proof of the theorem
with the 
\begin{claim*}
				$c\left( \absarg\le R \right)\in\mathcal{C}^{\star}$ is a $\star$-set.
				\label{claim:85}
\end{claim*}
\begin{proof}[Proof of claim:] 
				Write $c=c(\absarg\le R)=\left\{ c_{i} \right\}_{i=1}^{m-1}$ and let
				$c_{i}=v\in c\setminus \left( c^{\max} \amalg c^{\min} \right)$. We have, cf.
				(\ref{remark:53}), either 
				\begin{equation}
								\begin{array}[H]{l}
								f^{-}_{v}: n(c)_{i}^{+} = 0, \;\; n(c)^{-}_{i}=m>0 : c_{i+1} =
								(c_{i})_{+} = v_{+}, \;\; c_{i-1}=v_{-}+m\cdot v \\
								\text{or} \\
								f_{v}^{+}: n(c)_{i}^{+}=m>0, \;\; n(c)_{i}^{-}=0\,:\, c_{i+1}=v_{+}+m\cdot
								v\;\; , \;\; c_{i-1}=(c_{i})_{-} = v_{-}
								\end{array}
								\label{eq:86}
				\end{equation}
				If $m=1$ then $c_{i+1}-c_{i}=c_{i}-c_{i-1}$ and  $\left\{
				c_{i-1}<c_{i}<c_{i+1} \right\}$ are on a stright line, so assume $m>1$. We obtain
				either  
				\begin{equation}
								\begin{array}[H]{ll}
								\begin{array}[H]{ll}
								f_{v}^{-}: |c_{i+1}+c_{i}| &= |v_{+}+v| =
								\left| \begin{pmatrix} v \\ v_{+} \end{pmatrix} \right|  \le 
								\left| \begin{pmatrix} (m-1)v+v_{-} \\ v \end{pmatrix} \right| \\\\ &= |m\cdot v +
								v_{-}|
								=|c_{i-1}| 
								\le R
				\end{array}
								\\\\
								\text{a contradiction since $c_{i+1}+c_{i}\not\in c$ ; or}  \\\\
								\begin{array}[H]{ll}
								f_{v}^{+}: |c_{i}+c_{i-1}| &= |v+v_{-}| 
								= \left| \begin{pmatrix} v_{-} \\ v \end{pmatrix} \right| 
								\le \left| \begin{pmatrix} v \\ (m-1)v+v_{+} \end{pmatrix} \right| \\\\ &=
								|mv+v_{+}|  = |c_{i+1}| \le R 
								\end{array}
								\end{array}
								\label{eq:87}
				\end{equation}
				a contradiction since $c_{i}+c_{i-1}\not\in c$. 
\end{proof}
Note that 
\begin{equation}
				\begin{array}[H]{lll}
								v\in \Phi c(\absarg\le R) &\Longleftrightarrow &
								v+v_{+}, \;\; v+v_{-}\in c(\absarg\le R) \\\\
								&\Longleftrightarrow& |v|_{1}:= \sup \left\{ |v+v_{+}|, |v+v_{-}|
								\right\}\le R
				\end{array}
				\label{eq:88}
\end{equation}
The map $\absarg_{1}$ is again fundamentally  monotone so that by the claim: \\
$\Phi c(\absarg\le R)=c(\absarg_1\le R)\in
\mathcal{C}^{\star}$ is again a $\star$-set. \vspace{.2cm}\\
Denoting the \underline{Fibonachi numbers} by
\begin{equation}
a_{1}=a_{2}=1\;\; , \;\; a_{n}=a_{n-1}+a_{n-2}=\frac{1}{\sqrt{5}}\left[ \left(
\frac{1+\sqrt{5}}{2} \right)^n-\left( \frac{1-\sqrt{5}}{2} \right)^{n} \right]
				\label{eq:89}
\end{equation}
defined for $n\ge 0$,
\begin{equation}
				|v|_{n}:= \sup \left\{ |a_{n+2}v_{+}+a_{n+1}v_{-}|,\;
|a_{n+1}v_{+}+a_{n+2}v_{-}| \right\} \equiv |a_{n+2}M(v)+a_{n+1}F(v)|
				\label{eq:90}
\end{equation}
Then $\absarg_{n}$ is again fundamentally monotone, and from our claim we may
deduce inductively  that
\begin{equation}
				\Phi^n c\left( \absarg\le R \right) \equiv c\left(
				\absarg_{n}\le R \right)\in \mathcal{C}^{\star}
				\label{eq:91}
\end{equation}
so $c\left( \absarg\le R \right)$ is indeed a corona. We get (\ref{eq:91}) by
induction 
\begin{equation*}
				\begin{array}[H]{lll}
								v\in\Phi^{n}c(\absarg\le R)	&\Longleftrightarrow&
								v+v_{+}, v+v_{-}\in\Phi^{n-1} c \left( \absarg\le R \right)
								\\\\
								&\Longleftrightarrow& R  
								\begin{array}[t]{l}
								\ge \sup\left\{|v+v_{+}|_{n-1},|v+v_{-}|_{n-1} \right\}= \\\\
								= \sup \left\{ 
												\begin{array}[H]{l}
																% |a_{n+1}v_{+}+a_n v|,
																|a_{n}v_{+}+a_{n+1} v|, 
																|a_{n+1}v+a_{n}v_{-}|
																% , |a_{n}v+a_{n+1}v_{-}|
												\end{array}
								\right\} \\\\
								= \sup \left\{ 
												\begin{array}[H]{l}
																% |(a_{n+1}+a_n)v_{+}+a_{n}v_{-}|, \\
																|(a_{n+1}+a_{n})v_{+}+ a_{n+1} v_{-}|, \\\\
																|a_{n+1}v_{+}+(a_{n+1}+a_{n})v_{-}|
																% , \\  |a_{n}v_{+}+(a_{n+1}+a_{n})v_{-}| 
												\end{array}
								\right\} \\\\
								=\sup\left\{ |a_{n+2}v_{+}+a_{n+1}v_{-}|,
								|a_{n+1}v_{+}+a_{n+2}v_{-}| \right\} \\\\
								= |v|_{n}
								\end{array}
				\end{array}
\end{equation*}
This complete the proof of Theorem \ref{thm:84}. \vspace{.4cm}\\
Note that if $\absarg$ is pointwise-monotone, $v\prec v^{\prime}\Rightarrow
|v|\le |v^{\prime}|$, hence a-posteriori fundamentally monotone, the norms
$\absarg_{n}$, $n\ge 1$, need not be pointwise-monotone. \vspace{.2cm} \\
Let $\absarg$ be pointwise-monotone map $\absarg:\N\times
\N\rightarrow \R$ that is also \underline{homogenouse}: $|a\cdot v|=a\cdot |v|$,
$a\in \N^{+}$, then we have 
\begin{equation}
				|v|_{n}=\sup\left\{ |a_{n+2}v_{+}+a_{n+1}v_{-}|, |a_{n+1}v_{+}+a_{n+2}v_{-}|
				\right\} \le |a_{n+2}\cdot v|= a_{n+2}\bullet |v|
				\label{eq:92}
\end{equation}
Moreover, if $\absarg$ comes from a \underline{norm}, i.e. satisfies the
triangle inequality, we have for $v\in \Phi^{n}c(\absarg\le R)$
\begin{equation}
				a_{n+3}\bullet |v| = |a_{n+3}\cdot v|\le |a_{n+2}v_{+}+a_{n+1}v_{-}|+
				|a_{n+1}v_{+}+a_{n+2}v_{-}| \le 2\cdot R
				\label{eq:93}
\end{equation}
Together we have the \underline{exponential decay} of $\Phi^{n} c
(\absarg\le R)$. For a norm $\absarg$:
\begin{equation}
				c\left( \absarg \le \frac{1}{a_{n+2}}\cdot R \right)\subseteq
				\Phi^{n} c\left( \absarg \le R \right) \subseteq 
c\left( \absarg \le \frac{2}{a_{n+3}}\cdot R \right)
				\label{eq:94}
\end{equation}
Examples of such coronas are given by, $p\ge 1$, $R\ge 2$, 
\begin{equation}
				c_{R}^{(p)}:=c\left( x^{p}+y^{p}\le R^{p} \right)\quad , \quad c_{R}^{(\infty)}:=
				c\left( \max \left\{ x,y \right\}\le R \right)
				\label{eq:95}
\end{equation}
We have 
\begin{equation}
				c_{R}^{(1)}\subseteq c_{R}^{(2)}\subseteq c_{R}^{(\infty)}\subseteq c_{2R}^{(1)}
				\subseteq c_{2R}^{(2)} 
				\subseteq c_{2R}^{(\infty)} 
				\subseteq c_{4R}^{(1)} \subseteq \cdots
				\label{eq:96}
\end{equation}
We have as well the norms $\absarg_{A}$, for a positive real matrix $A=
\begin{pmatrix}
				a_{1 1} & a_{1 2} \\
				a_{2 1} & a_{2 2} \\
\end{pmatrix}
$, $a_{i j}>0$, 
\begin{equation}
				\left| g\right|_{A} = \left|
\begin{pmatrix}
				x_{-} & y_{-} \\
				x_{+} & y_{+}
\end{pmatrix}
\right|_{A} := \text{tr}\left( g\cdot A^{t} \right) = x_{-}\cdot a_{1 1}+y_{-}\cdot
a_{1 2}+x_{+}\cdot a_{2 1}+y_{+}\cdot a_{2 2}.
				\label{eq:97}
\end{equation}
In particular taking $A=
\begin{pmatrix}
				\alpha & \beta \\
				\alpha & \beta
\end{pmatrix}
$, we have the \underline{linear-norms}, $\alpha,\beta>0$, 
\begin{equation}
				\begin{array}[H]{ll}
				|x,y|_{(\alpha,\beta)} := \alpha \cdot x+\beta\cdot y \\\\
				|v|_{(\alpha,\beta)} = |v_{+}+v_{-}|_{(\alpha,\beta)}=
				|v_{+}|_{(\alpha,\beta)}+|v_{-}|_{(\alpha,\beta)} \\\\
				|vg|_{(\alpha,\beta)}=\text{tr}\left( g_{v}\cdot  g\cdot 
												\begin{pmatrix} \alpha & \beta \\ \alpha & \beta
												\end{pmatrix}^{t}
				\right) =
\text{tr} \left( g_{v}\cdot \left(  
\begin{pmatrix} \alpha & \beta \\ \alpha & \beta \end{pmatrix}
				\cdot g^{t}\right)^t\right) = |v|_{(\alpha,\beta)g^{t}}
				\end{array}
				\label{eq:98}
\end{equation}
Thus we have the coronas 
\begin{equation}
				c_{R}^{(\alpha,\beta)}:= \left\{ (x,y)\in \underline{\Q}^{+}, \alpha\cdot
				x+\beta\cdot y\le R \right\}
				\label{eq:99}
\end{equation}
We have for $g\in \SL_{2}(\N)$, with $\underline{0}g$, $\underline{\infty}g\in
c_{R}^{(\alpha,\beta)}$, 
\begin{equation}
				c_{R}^{(\alpha,\beta)}\cap \left( \underline{0}g,\underline{\infty}g \right) =
				\left( c_{R}^{(\alpha,\beta)g^{t}} \right)g
				\label{eq:100}
\end{equation}
\section{The d.n.a of $c_{R}^{(\alpha,\beta)}$}
Fix a linear-norm $\absarg=\absarg_{(\alpha,\beta)}$,
$\alpha,\beta>0$, and $R\in\R$, $n\ge 1$, and let 
\begin{equation}
				\begin{array}[H]{ll}
				c^{n-1}=\Phi^{n-1}c_{R}^{(\alpha,\beta)}=\left\{ v\in\Q^{+},
				a_{n+1}|v_{+}|+a_{n}|v_{-}|, a_{n}|v_{+}|+a_{n+1}|v_{-}|\le R \right\} \\\\
				c^{n}=\Phi^{n}c_{R}^{(\alpha,\beta)}=\left\{ v\in\Q^{+},
				a_{n+2}|v_{+}|+a_{n+1}|v_{-}|, a_{n+1}|v_{+}|+a_{n+2}|v_{-}|\le R \right\} \\\\
				c^{n-1}=c^{n}\circ \lambda^{n} \\ 
				\lambda^{n}=\sum\limits_{i=1}^{m}\lambda_{i}^{n}\left[ c_{i-1},c_{i} \right]\in
				\Z\partial c^{n}\;\; , \;\; c^{n}=\left\{ c_{i} \right\}_{i=1}^{m-1}.
				\end{array}
				\label{eq:101}
\end{equation}
Let $v\in \left( c^{n-1}\right)^{\max}$, and let $c_{i}>c_{i-1}\in \left( c^{n-1}
\right)^{\min} $ be the points imediately above and below it, $c_{i}>v>c_{i-1}$, we have
via (\ref{remark:62}-\ref{eq:63}), that there exists $k_{1}>k_{0}\ge 0$ such that either 
\begin{equation}
								\begin{array}[H]{lll}
																\left( f^{+}_{c_{i-1}} \right): c^{n-1}\cap \left[ c_{i-1},c_{i} \right]=
												\left\{
												\begin{array}[H]{ll}
																c_{i}&=\left( c_{i-1} \right)_{+}+k_{0}c_{i-1} > \cdots 
																> v \\\\
															&=\left( c_{i-1} \right)_{+}+k_{1}c_{i-1}>c_{i-1}
												\end{array} 
												 \right\},
\\\\
\lambda^{n}_{i}=-\left( k_1-k_0 -1\right); \text{ or } \\\\
\left( f_{c_i}^{-} \right): c^{n-1}\cap \left[ c_{i-1},c_{i} \right]= \left\{
				\begin{array}[H]{lll}
				c_{i} > v &=\left( c_{i} \right)_{-}+k_{1}c_{i}>\cdots >  
				\left( c_{i}
				\right)_{-}+k_{0}c_{i} \\\\
				&=c_{i-1} 
				\end{array}
\right\}, \\\\
\lambda_{i}^{n} = \left( k_{1}-k_{0}-1 \right). 
				\end{array}
				\label{eq:102}
\end{equation}
Thus for the case of $\left( f_{c_{i-1}}^{+} \right)$ we have 
\begin{equation}
				\begin{array}[H]{lll}
				c_{i}=\left( c_{i-1} \right)_{+}+k_{0}\cdot c_{i-1}\in c^{n} & , & 
				c_{i}+c_{i-1} = \left( c_{i-1} \right)_{+}+\left( k_{0}+1 \right)c_{i-1}\not\in
				c^{n} \\\\
				v=\left( c_{i-1} \right)_{+}+k_{1}\cdot c_{i-1}\in c^{n-1} & , &
				v+c_{i-1}=(c_{i-1})_{+}+(k_1+1)\cdot c_{i-1}\not\in c^{n-1}
				\end{array}
				\label{eq:103}
\end{equation}
\begin{equation}
				\Longleftrightarrow
				\begin{array}[t]{l}
				\left|\left( c_{i-1}
				\right)_{+}+k_{0}c_{i-1}\right|_{n}\le R < \left|\left( c_{i-1} \right)_{+}+
				\left( k_{0}+1 \right)c_{i-1}\right|_{n} \\\\
				\left|\left( c_{i-1} \right)_{+}+k_{1}c_{i-1}\right|_{n-1}\le R < \left|
				\left( c_{i-1} \right)_{+}+\left( k_1 +1 \right)c_{i-1}\right|_{n-1}
				\end{array}
				\label{eq:104}
\end{equation}
\begin{equation}
				\begin{array}[H]{l}
				\Longleftrightarrow  \\\\
				\begin{array}[t]{l}
				a_{n+2}\left|\left( c_{i-1} \right)_{+}+ \left( k_{0}-1
				\right)c_{i-1}\right|+a_{n+1}\left| c_{i-1}\right| \le R <
				\ \hspace{-.5cm} \			\begin{array}[t]{ll}
				a_{n+2}\big|\left(
				c_{i-1}\right)_{+} \hspace{-.5cm} \ &+k_{0}c_{i-1}\big|  \vspace{.1cm}\\ 
				&+a_{n+1}\left|c_{i-1}\right| 
								\end{array}
				 \\\\ 
				 a_{n+1}|(c_{i-1})_{+}+(k_1-1)c_{i-1}|+a_n|c_{i-1}|\le R <
				 \ \hspace{-.5cm} \ \begin{array}[t]{ll}
								 a_{n+1}|(c_{i-1})_{+}\hspace{-.5cm} \ &+k_1c_{i-1}| \vspace{.1cm}\\ &+a_n|c_{i-1}| 
				 \end{array}
				\end{array}
				\end{array}
				\label{eq:105}
\end{equation}
Thus $k=|\lambda_i^n|=k_1-k_0-1$ is the maximal integer such that 
\begin{equation*}
				(a_{n+1}\cdot k+a_n)\cdot|c_{i-1}|+a_{n+1}|c_i| =
				a_{n+1}|(c_{i-1})_{+}+(k_1-1)c_{i-1}|+a_n|c_{i-1}| \le R
\end{equation*}
and we obtain, $c_i=k_0 c_{i-1}+(c_{i-1})_{+}$, 
\begin{equation}
				\begin{array}[H]{lll}
								\left|\lambda_{i}^{n}\right| &=
								\left\lfloor\frac{R-a_{n+1}|c_{i}|-a_{n}|c_{i-1}|}{a_{n+1}|c_{i-1}|}
								\right\rfloor \\\\
&= 								
\left\lfloor\frac{R-a_{n+1}|(c_{i-1})_{+}|-a_{n}|c_{i-1}|}{a_{n+1}|c_{i-1}|}
\right\rfloor  -k_{0},\quad \text{and when $k_{0}>0$:} \\\\
\left|\lambda_{i}^{n}\right| &=  \left\lfloor \frac{R}{a_{n+1}|c_{i-1}|} -
\frac{|(c_{i-1})_{+}|}{|c_{i-1}|}-\frac{a_n}{a_{n+1}} \right\rfloor - 
\left\lfloor \frac{R}{a_{n+2}|c_{i-1}|}-\frac{|(c_{i-1})_{+}|}{|c_{i-1}|}-\frac{a_{n+1}}{a_{n+2}}
\right\rfloor-1
				\end{array}
				\label{eq:106}
\end{equation}
Similarly in the case of $\left( f_{c_{i}}^{-} \right)$ we have, $c_{i-1}=k_0
c_i+(c_i)_{-}$, 
\begin{equation}
				\begin{array}[H]{lll}
				\lambda_{i}^{n}
				&= \left\lfloor \frac{R-a_{n+1}|c_{i-1}|-a_{n}|c_{i}|}{a_{n+1}|c_{i}|}  \right\rfloor \\\\
				&= \left\lfloor \frac{R-a_{n+1}|(c_{i})_{-}|-a_{n}|c_{i}|}{a_{n+1}|c_{i}|}
				\right\rfloor -k_0 \quad, \quad \text{and when $k_0>0$:} \\\\
				&= \left\lfloor \frac{R}{a_{n+1}|c_i|}-\frac{|(c_i)_{-}|}{|c_{i}|} -
				\frac{a_n}{a_{n+1}} \right\rfloor - \left\lfloor \frac{R}{a_{n+2}|c_i|}-\frac{|(c_i)_{-}|}{|c_{i}|} -
				\frac{a_{n+1}}{a_{n+2}} \right\rfloor -1
				\end{array}
				\label{eq:107}
\end{equation}
Note that, in the case of $\left( f_{c_{i-1}}^{+} \right)$, we have
$\lambda_{i}^{n}\neq 0$ iff 
\begin{equation}
				\begin{array}[H]{ccc}
				\ & \frac{R-a_{n+1}|c_i|-a_n|c_{i-1}|}{a_{n+1}|c_{i-1}|}\ge 1 \\\\
				\Longleftrightarrow  & a_{n+2}|c_{i-1}|+a_{n+1}|c_i| \le R
				\end{array}
				\label{eq:108}
\end{equation}
We cannot have $k_{0}=0$, because than $c_{i}=\left( c_{i-1} \right)_{+}$, $k_{1}\ge 2$,
so that \[ 
				\begin{array}[h]{lll}
								a_{n+1}|(c_{i-1})_{+}+c_{i-1}|+a_{n}|c_{i-1}| &=& \left|(c_{i-1})_{+}+2\cdot
c_{i-1}\right|_{n-1}\le R \\\\
&<&
\left|(c_{i-1})_{+}+c_{i-1}\right|_{n}\\\\
&=& a_{n+2}|c_{i-1}|+a_{n+1}\left|(c_{i-1})_{+}\right|
				\end{array}
\]
a contradiction. Therefore $k_{0}\ge 1$, and so
$k_1=|\lambda_{i}^{n}|+k_0+1\ge 3$. We see that in this case $c_{i-1}\in \left(
				c^{n}
\right)^{\min}$, that is $c_{i-1}+(c_{i-1})_{-}\in c^{n}$, for otherwise we get 
\begin{equation}
				\begin{array}[H]{lll}
				a_{n+2}|c_{i-1}|+a_{n+1}|(c_{i-1})_{-}| &=  |c_{i-1}+(c_{i-1})_{-}|_{n} \\\\
				&> R \\\\
				&\ge |3\cdot c_{i-1}+(c_{i-1})_{+}|_{n-1} \\\\
				&= a_{n+1}|2\cdot c_{i-1}+(c_{i-1})_{+}|+a_{n}|c_{i-1}| \\\\
				&= a_{n+2}|c_{i-1}|+a_{n+1}|c_{i-1}+(c_{i-1})_{+}|
				\end{array}
				\label{eq:109}
\end{equation}
a condradiction. \\
Similarly in the case of $(f_{c_{i}}^{-})$ we have $\lambda_{i}^{n}\neq 0$ iff 
\begin{equation}
				a_{n+2}|c_{i}|+a_{n+1}|c_{i-1}| \le R
				\label{eq:110}
\end{equation}
and this imply $|c_{i-1}|>|c_{i}|$, $k_{0}\ge 1$, $k_{1}\ge 3$, and $c_{i}\in \left(
				c^{n}
\right)^{\min}$. \\
We summerize this in the following decription of the ``d.n.a. of $c_{R}^{(\alpha,\beta)}$
'': 
\begin{theorem}
				For a linear norm $|x,y|:=\alpha x+\beta y$, $\alpha,\beta>0$, we have for
				$n\ge 1$, 
				\begin{equation*}
								\begin{array}[H]{lll}
												\Phi^{n-1}c_{R}^{(\alpha,\beta)}=
												\Phi^{n}c_{R}^{(\alpha,\beta)}\circ \lambda_{R}^{n} \\\\
												\lambda_{R}^{n}=\sum\limits_{i=1}^{m}\lambda_{i}^{n}\cdot
												[c_{i-1},c_{i}]\in \Z\partial\Phi^{n}c_{R}^{(\alpha,\beta)}\quad ,
												\quad \Phi^{n}c_{R}^{(\alpha,\beta)}=\left\{ c_{i}
												\right\}_{i=1}^{m-1}, 
								\end{array}
				\end{equation*}
				and 
				\begin{equation*}
								\begin{array}[H]{lll}
												\lambda_{R}^{n}=\lambda_{\underline{\infty}}^{-}\cdot\left[ c_{m-1},\underline{\infty}
												\right]+\sum\limits_{c_{i}\in\Phi^{n+1}c_{R}^{(\alpha,\beta)}} \left(
												\lambda_{c_{i}}^{+}\cdot \left[ c_{i},c_{i+1}
												\right]+\lambda_{c_{i}}^{-}[c_{i-1},c_{i}] \right)+
												\lambda_{\underline{0}}^{+}\cdot \left[ \underline{0},c_{1} \right] \\\\
												|\lambda_{c_i}^{+}| =
												\begin{cases}
																					\left\lfloor\frac{R}{a_{n+1}|c_{i}|}-\frac{|(c_i)_{+}|}{|c_i|} -
																				\frac{a_{n}}{a_{n+1}}\right\rfloor -  
																				\left\lfloor
																				\frac{R}{a_{n+2}|c_{i}|}- \frac{|(c_{i})_{+}|}{|c_{i}|} -
																				\frac{a_{n+1}}{a_{n+2}}\right\rfloor  - 1,
																				\\\\
																				0  \qquad \text{ if }\qquad  a_{n+2}|c_{i}|+a_{n+1}|c_{i+1}|>R;
																\end{cases} \\\\
												\lambda_{c_i}^{-} =
												\begin{cases}
																\left\lfloor\frac{R}{a_{n+1}|c_{i}|}-\frac{|(c_i)_{-}|}{|c_i|} -
																\frac{a_{n}}{a_{n+1}}\right\rfloor -  \left\lfloor
																\frac{R}{a_{n+2}|c_{i}|}- \frac{|(c_{i})_{-}|}{|c_{i}|} -
																\frac{a_{n+1}}{a_{n+2}}\right\rfloor  - 1, \\\\
																$0$ \qquad \text{ if } \qquad a_{n+2}|c_{i}|+a_{n+1}|c_{i-1}|>R;
												\end{cases}
												\\\\
												|\lambda_{\underline{0}}^{+}| =
												\begin{cases}
																\left\lfloor
																\frac{R}{a_{n+1}\alpha}-\frac{\beta}{\alpha}-\frac{a_n}{a_{n+1}}\right\rfloor
																- 
																\left\lfloor\frac{R}{a_{n+2}\alpha}-\frac{\beta}{\alpha}-\frac{a_{n+1}}{a_{n+2}}\right\rfloor
																- \underline{1}, \\\\
																0 \qquad \text{ if } \qquad
																a_{n+2}\alpha+a_{n+1}|c_1|>R;
												\end{cases}
												\\\\
												\lambda_{\underline{\infty}}^{-}=
												\begin{cases}
																\left\lfloor
																\frac{R}{a_{n+1}\beta}-\frac{\alpha}{\beta}-\frac{a_n}{a_{n+1}}\right\rfloor -
																 \left\lfloor
																\frac{R}{a_{n+2}\beta}-\frac{\alpha}{\beta}-\frac{a_{n+1}}{a_{n+2}}\right\rfloor
																-1  , \\\\ 
																0 \qquad \text{ if } \qquad
																a_{n+2}\beta+a_{n+1}|c_{m-1}|>R.
												\end{cases}
								\end{array}
				\end{equation*}
				\label{thm:111}
\end{theorem}
Thus if $\Phi^{n} c_{R}^{(\alpha,\beta)}=\left\{ c_{i} \right\}_{i=1}^{m-1}$, we obtain
$\Phi^{n-1}c_{R}^{(\alpha,\beta)}$ from it by adding all the mediants $c_{i}+c_{i-1}$, and
for those $c_{i}\in \Phi^{n+1}c_{R}^{(\alpha,\beta)}$, as well as
$c_{m}=\underline{\infty}$, $c_{0}=\underline{0}$, we add the fin around $c_{i}$ whose
length is given by the $\lambda_{c_i}^{\pm}$. 
\begin{corollary*}
				We have, with $\lfloor R\rfloor_{+}:=\max\left\{ 0,\lfloor R\rfloor \right\}$,
				\begin{equation*}
								\# \Phi^{n-1}c_{R}^{(\alpha.\beta)} = 
								\begin{array}[t]{lll}
												2\cdot \#\Phi^{n} c_{R}^{(\alpha,\beta)}  \\\\
												+ \hspace{-.8cm} \lsum\limits_{c_i=v\in \Phi^{n+1}
												c_{R}^{(\alpha,\beta)}} \hspace{-.6cm}
												\begin{array}[t]{l}
												\left\lfloor\frac{R}{a_{n+1}|v|}-\frac{|c_{i+1}|}{|v|}-\frac{a_{n}}{a_{n+1}}\right\rfloor_{+}+
												% \\\\
												\left\lfloor
												\frac{R}{a_{n+1}|v|}-\frac{|c_{i-1}|}{|v|}-\frac{a_n}{a_{n+1}}\right\rfloor_{+}
												\end{array}
												\\\\
												+\left\lfloor
								\frac{R}{a_{n+1}\alpha}-\frac{|c_1|}{\alpha}-\frac{a_n}{a_{n+1}}\right\rfloor_{+}
												+ \left\lfloor
												\frac{R}{a_{n+1}\beta}-\frac{|c_{m-1}|}{\beta}-\frac{a_n}{a_{n+1}}\right\rfloor_{+}
								\end{array}
				\end{equation*}
				\label{cor:112}
\end{corollary*}
The formula of Theorem \ref{thm:111} show a kind of interaction between the binary and
Fibonachi bases. \vspace{.1cm}\\
Recall that every integer $R\in \N$ has a binary expansion 
\begin{equation}
				R=2^{n_1}+\cdots + 2^{n_{j}}+\cdots + 2^{n_{\ell}}\quad , \quad n_{j}\ge 0.
				\label{eq:113}
\end{equation}
This expansion is unique if we require $n_j>n_{j+1}$. We can add such numbers and bring
them to the canonical form using the ``carry-reminders'' rule $2^{n}+2^{n}=2^{n+1}$. We
can multiply numbers using the simple rule $2^{n}\cdot 2^{m}=2^{n+m}$. \\ 
Rewriting the
Fibonacci numbers as 
\begin{equation}
				\begin{array}[H]{lll}
				\varphi^{n}:= a_{1+n} &= \frac{1}{\sqrt{5}}\left[ \left( \frac{1+\sqrt{5}}{2}
				\right)^{n+1}- \left( \frac{1-\sqrt{5}}{2} \right)^{n+1} \right] \\\\
				&=\frac{1}{2^{n}}\sum\limits_{k=0}^{n}(1+\sqrt{5})^{k}(1-\sqrt{5})^{n-k}
				\end{array}
				\label{eq:114}
\end{equation}
Similarly, every $R\in \N$ has a Fibonacci or Zeckendorf expansion as a sum 
\begin{equation}
				R=\varphi^{n_1}+\cdots+\varphi^{n_j}+\cdots+ \varphi^{n_{\ell}}
				\label{eq:115}
\end{equation}
This expansion is unique if we require that $n_j>n_{j+1}+1$, i.e. we can represent
$R$ by a sequence of $0$'s and $1$'s, where no two $1$'s are neighbours. We can add
numbers in this representation, and bring them to the cannonical form using the
``carry-reminder'' rules: $\varphi^n+\varphi^{n+1}=\varphi^{n+2}$, and
$\varphi^{n}+\varphi^{n}=\varphi^{n+1}+\varphi^{n-2}$. We can also multiply numbers using
the rule, for $m\ge n$: 
\begin{equation*}
				\begin{array}[H]{lll}
				\left( \star \right)_{n,m} \varphi^{n}\cdot
				\varphi^{m} &=& \varphi^{n+m}+\varphi^{n+m-4}+\cdots+\varphi^{n+m-4j}+\cdots  \\\\
				&\ & + \left\{
								\begin{array}[H]{ll}
												\varphi^{m-n+4}+\varphi^{m-n} & n\equiv 0(2) \\\\
												\varphi^{m-n+2}+\varphi^{m-n-1} & n\equiv 1(2)\;\; n<m \\\\
												\varphi^{2}+\varphi^{0} & m=n\equiv 1(2)
								\end{array}
				\right.
				\end{array}
\end{equation*}
One prove $(\star)_{n,m}$ by induction, via
\begin{equation*}
				\begin{array}[H]{l}
								(\star)_{n,n}+(\star)_{n-1,n} \Longrightarrow (\star)_{n,n+1} \\\\
								(\star)_{n-1,n}+(\star)_{n-2,n} \Longrightarrow (\star)_{n,n} \\\\
								(\star)_{n,n}+(\star)_{n,n+1} \Longrightarrow (\star)_{n,m} \quad, \quad
								m\ge n. 
				\end{array}
\end{equation*}
Note the curious $4$-periodicity of $(\star)_{n,m}$. \vspace{.1cm} \\ 
% The following monotone functions
% $f_{\varepsilon}:[0,\infty]\longrightarrow \N$, $\varepsilon\in[0,1)$ seem relevant (as
% they appear when we sum $|\lambda_{v}^{n,\pm}$ over $n\ge 0$, 
% with $R:=R/|v|$, and $\varepsilon = |v_{\pm}|/|v|$):
% \begin{equation}
% 				f_{\varepsilon}(R):= \lsum\limits_{
% 								\overset{\scriptscriptstyle n\ge 0 }{\scriptscriptstyle R> \varphi^n+\varepsilon\cdot \varphi^{1+n}
% 				}}
% 																2^{n}\left\lfloor
% 				\frac{R-\varphi^n}{\varphi^{1+n}}-\varepsilon\right\rfloor 
% 				\label{eq:116}
% \end{equation}
Perhaps this interaction of the \underline{binary} and \underline{Fibonacci} expansions should come as
no surprise since our very approach to $\Q^{+}$ is as a \underline{binary} tree of
\underline{Fibonacci} growth. 
\section{Equidistribution} 
				We end with some remarks on equidistribution. \\
				We have the exact
				\underline{potential function} 
				\begin{equation}
								\begin{array}[H]{lll}
								h:\fg_{1}=\SL_2(\N)\rightarrow [0,1] \\\\
								h(v_{-},v_{+}) := \frac{1}{|v_{-}|\cdot |v_{+}|}, \quad ,\quad
								|(x,y)|=x+y.
								\end{array}
								\label{eq:117}
				\end{equation}
For each triangle $\Delta_{v}$ we have the \underline{exactness}
\begin{equation}
				h(v_{-},v_{+}) = h\left( v_{-},v \right)+h\left( v,v_{+} \right)
				\label{eq:118}
\end{equation}
We get the function
\begin{equation}
				\begin{array}[H]{l}
								H:\left\{ \underline{\infty} \right\}\amalg \underline{\Q}^{+} \amalg
								\left\{ \underline{0} \right\}=\fg_{0}\xrightarrow{\qquad} [0,1] \\\\
								\displaystyle H(v)=\displaystyle\int\limits_{\underline{0}}^v h\left( dg \right) = \text{sum of
								$h$ along (any) path from $\underline{0}$ to $v$.}
				\end{array}
				\label{eq:119}
\end{equation}
We have 
\begin{equation}
				H(x,y)=\frac{y}{x+y}
				\label{eq:120}
\end{equation}
Indeed, 
\begin{equation}
				\begin{array}[H]{lll}
				\displaystyle \partial H 
				\begin{pmatrix}
								x_{-} & y_{-} \\
								x_{+} & y_{+}
				\end{pmatrix} &= \displaystyle
				\frac{y_{+}}{x_{+}+y_{+}} - \frac{y_{-}}{x_{-}+y_{-}} \\\\ 
				&= \displaystyle
				\frac{1}{(x_{+}+y_{+})(x_{-}+y_{-})}\\\\
				&\displaystyle= \displaystyle h 
				\begin{pmatrix}
								x_{-} & y_{-} \\
								x_{+} & y_{+}
				\end{pmatrix}
				\end{array}
				\label{eq:121}
\end{equation}
For a $\lideal$-set $c=\left\{ c_{i} \right\}^{m-1}_{i=1}\in\mathcal{C}_{m}$, we get the
function 
\begin{equation}
				\begin{array}[H]{l}
				R_{c}: \left\{ \underline{\infty} \right\} \amalg c \amalg \left\{
								\underline{0}
				\right\} \xrightarrow{\qquad} [0,1] \\\\
				R_{c}(v) = \left\{ 
								\begin{array}[H]{ll}
												1 & v=c_{m}=\underline{\infty} \\
												j/m & v=c_j \\
												0 & v=c_{0}=\underline{0}
								\end{array}
				\right\} = \frac{1}{m}\displaystyle\int_{\underline{0}}^{v}\unit
				\end{array}
				\label{eq:122}
\end{equation}
the length of the path $c$ from $\underline{0}$ to $v$ divided by the total length of c.
\\
Put for $p\ge 1$, 
\begin{equation}
				\delta_{p}(c)=\| R_{c}-H\|^{p}_{\ell_{p(c)}} = \sum\limits_{j=1}^{m-1} \left|
				\frac{j}{m} -\frac{y_j}{x_{j}+y_{j}}\right|^{p} \quad , \quad c=\left\{
								c_{j}=(x_{j},y_{j})
				\right\}.
				\label{eq:123}
\end{equation}
For $c=c_{R}=c_{R}^{(1,1)}=\left\{ (x,y)\in\Q^{+}, x+y\le R \right\}$, we have that the
following estimates imply Riemann Hypothesis, 
\begin{equation}
				\text{Franel \cite{F1924}:} \qquad \delta_{2}(c_{R})=O\left( \frac{\log R}{R} \right)
				\label{eq:124}
\end{equation}
\begin{equation}
				\text{Landau \cite{L1924}:} \qquad \delta_{1}(c_{R})= O\left( R^{1/2}\log R \right)
				\label{eq:125}
\end{equation}
To obtain this using an inductive procedure, one will need a good estimation of
$\delta_{p}\left( \Phi^{n-1}c_{R} \right)$ in terms of $\delta_{p}\left( \Phi^{n}
c_{R} \right)$. One can try to do this ``locally'', by dissecting $\Phi^{n}c_{R}$ to
intervals. \underline{A partial $\lideal$-set}, or a $\lideal$-\underline{interval}, is a
path $c=\left\{ c_{i} \right\}_{i=0}^{m}$ in the Farey Graph, $\det 
\begin{pmatrix}
				c_{i-1} \\ c_{i}
\end{pmatrix} \equiv 1
$, from the \underline{initial-point} $c_0$ to the \underline{end-point} $c_{m}$
(and similarly one can define a \underline{partial} $\star$-\underline{set} and
\underline{partial corona}). For such $\lideal$-interval $c=\left\{
c_{i}=(x_{i},y_{i}) \right\}_{i=0}^{m}$ we can define
\begin{equation}
				\delta_{p}(c):= \sum\limits_{i} \left|
				\frac{y_{0}}{x_{0}+y_{0}}+\frac{i}{m}\left(
				\frac{y_{m}}{x_{m}+y_{m}}-\frac{y_{0}}{x_{0}+y_{0}} \right) -
				\frac{y_{i}}{x_{i}+y_{i}}\right|^{p}
				\label{eq:126}
\end{equation}
This agree with (\ref{eq:123}) when $(x_{0},y_{0})=\underline{0}=(1,0)$.
$(x_{m},y_{m})=\underline{\infty}=(0,1)$. One can also demand that $\det 
\begin{pmatrix}
				c_{0} \\ c_{m}
\end{pmatrix} = 1
$, so that 
\[\left(
\frac{y_{m}}{x_{m}+y_{m}}-\frac{y_{0}}{x_{0}+y_{0}}\right)=\frac{1}{(x_{m}+y_{m})\cdot
(x_0+y_0)},\]
and $c=\left\{ c_i=\check{c}_{i} \begin{pmatrix} c_0 \\ c_m \end{pmatrix}
\right\}$ where $\check{c}=\left\{ \check{c}_{i} \right\}$ is a usual $\lideal$-set
(or corona). \vspace{.1cm}\\ E.g. 
For such a partial $\lideal$-set (or corona) $c=\left\{ c_{i}
\right\}_{i=0}^{m}$ one has the associated partial $\lideal$-set (corona) $\tilde{c}:= c\cup
\left\{ c_{i}+c_{i-1} \right\}_{i=1}^{m}$ with the same initial and end points, obtained
by adding all medians. There is an elementary estimation
\begin{equation}
				\begin{array}[H]{l}
				\delta_{1}(\tilde{c})\le 2\cdot \delta_{1}(c)+\frac{1}{2} h(c) \\\\
				h(c) = \frac{y_{m}}{x_{m}+y_{m}} - \frac{y_{0}}{x_{0}+y_{0}}\in [0,1] \qquad
				\text{the real length of $c$.}
				\end{array}
				\label{eq:127}
\end{equation}
Along the fins one can estimate $\delta_1$ using the Euler-MacLaurin formula.
Also, if $c=\coprod\limits_{j} c_{j}$, where the end point of $c_{j-1}$ is the initial
port of $c_{j}$, we have the elementary estimate 
\begin{equation}
				\begin{array}[H]{l}
				\left| \delta_{1}(c)- \sum\limits_{j} \delta_{1}(c_{j})\right| \le
				\frac{1}{2}\sum\limits_{j_{1}<j_{2}} \left|
				h(c_{j_{1}})M(c_{j_{2}})-h(c_{j_{2}})M(c_{j_{1}}) \right| \\\\
M(c) = \# c+1 = m\in \N \qquad \text{the degree or length of the path } \partial c.
				\end{array}
				\label{eq:128}
\end{equation}
\ \vspace{.2cm}\\
But it is important to note that $H(c_{R})\subseteq \left[ 0,1 \right]$ is
\underline{Not} equidistributed, it is only on average so (\ref{eq:124}-\ref{eq:125}): the
real distance between $H(v)=H(c_{i})$ and $H(c_{i\pm 1})$, for an ``old'' $v$, so
$c_{i\pm 1}\in f_{v}^{\pm}$, is of the order $O\left( \frac{1}{|v|\cdot (R+|v_{\pm}|)}
\right)$; while the real distance between the elements of $f^{+}_{v}$, or of
$f_{v}^{-}$, are smaller- 
\[ \left| H(c_{i\pm 1})- H(c_{i\pm 2})\right| = O\left(\frac{1}{R\cdot (R+|v|)}
\right) .\]
E.g. for $v=\underline{0}=(1,0)=c_{0}$, we have $c_{1}=(R,1)$,
$c_{2}=(R-1,1)$, and $|H(c_1)-H(c_0)|=\frac{1}{(R+1)}$, while 
$|H(c_2)-H(c_1)| = \frac{1}{R}-\frac{1}{R+1}= \frac{1}{R\cdot (R+1)}$. Thus it is
important that the d.n.a. of $\Phi^n c_R$ adds extra points just around such old
$v$-s, cf. Theorem \ref{thm:111}.
\begin{remark}
  \label{10.13}
  Let $c(n)$ denote the corona of height $n$, with d.n.a. identically
  $0$, so that $c(n+1)$ is obtained from $c(n)$ just by adding all
  mediants, and $\# c(n)=2^n-1$, so
  $c(n) = \left\{(x_i,y_i)\right\}_{i=1}^{2^n-1}$ in increasing real
  order, and let
  \begin{equation}
    \label{eq:10.14}
    S_n = \sum_{i=1}^{2^n-1}\left|\frac{i}{2^n}-\frac{y_i}{x_i+y_i}\right|^2.
  \end{equation}
  The sum $S_n$ appears in all the even places of the sum $S_{n+1}$,
  so that
  \[S_1=0<S_2=\frac{2}{144}<S_3=\frac{668}{14400} < \dots < \dots \]
  is monotone increasing, and does not converge to $0$. Comparing this to
  (\ref{eq:124}), we see that it is the d.n.a. that is responsible
  for the uniform distribution of the rationals within the continuum.
\end{remark}

%%%%%%%%%%%%%%%%%%%%%%%%%%%%%%%%%%%%%%%%%%%%%%%%%%%%%%%%%%%%%%%%%%%%%%%%%%%%%%%%%%%%%%%%%%%%%%%%%%%%%%%%%%%%%%%%%%%%%
		 % bibliography

\bibliographystyle{unsrt}
\bibliography{fib}
\end{document}

\begin{figure}[h]
				\centering
				\begin{tikzpicture}
								\usetikzlibrary{decorations.markings}
								\node  (vplus) at (-3.8cm,0cm) {};
								% \draw (vplus) circle  [radius=2pt];
								\fill[] ($(vplus)$) circle (2pt);
								\node  at ($(vplus)$) {$v_{+}$};
								\node (v) at (0cm,2.5cm) {$v$};
								\node (vm) at (-0.5cm,1cm) {$v_{-}$};
								\node (empty) at (-1.2cm,0cm) {.};
								\node (mp) at (-1.3cm,-1.5cm) {$M^{m}(v)$};
								\node (vminus) at (5cm,0cm) {$v_{-}$} ;
								\draw ($(vplus)$)--(v);
								\draw[postaction={decorate},{decoration={markings,mark=at position 0.5 with
								{\arrow[color=black]{>}}}},color=green] (vplus)--(v); 
				\end{tikzpicture}
				\caption{ Concavity and convexity of $\delta(v)$ }
				\label{fig:2}
\end{figure}